\documentclass[10pt]{amsart}

\usepackage{times}
\usepackage{amsmath,amssymb}
\usepackage{hyperref}

\usepackage{geometry}
\newtheorem{theorem}{Theorem}
\newtheorem{lemma}{Lemma}

\newtheorem{definition}{Definition}
\newtheorem{proposition}{Proposition}
\newtheorem{conjecture}{Conjecture}
\newtheorem{corollary}{Corollary}

\def\s{c}
\def\ar{\mathcal{T}_\alpha}

\begin{document}

\title{On {$\mathbf{\alpha}$}-roughly weighted games}

\author{Josep Freixas}
\author{Sascha Kurz}
\address{Josep Freixas\\Department of Applied Mathematics III and Engineering School of Manresa\\Universitat
Polytecnica de Catalunya\\Spain\footnote{This author acknowledges the support  of the Barcelona Graduate School
of Economics and of the Government of Catalonia.}}
\email{josep.freixas@upc.edu}
\address{Sascha Kurz\\Department for Mathematics, Physics and Informatics\\University Bayreuth\\Germany}
\email{sascha.kurz@uni-bayreuth.de}


\maketitle

\begin{abstract}
  Gvozdeva, Hemaspaandra, and Slinko (2011) have introduced three hierarchies for simple games in order to
  measure the distance of a given simple game to the class of (roughly) weighted voting games. Their
  third class $\mathcal{C}_\alpha$ consists of all simple games permitting a weighted representation such that each winning
  coalition has a weight of at least $1$ and each losing coalition a weight of at most $\alpha$. For a given
  game the minimal possible value of $\alpha$ is called its critical threshold value.  We continue the work 
  on the critical threshold value, initiated by Gvozdeva et al., and contribute some new results on the possible values
  for a given number of voters as well as some general bounds for restricted subclasses of games. A
  strong relation beween this concept and the cost of stability, i.e.\ the minimum amount of external payment to ensure
  stability in a coalitional game, is uncovered.
\end{abstract}

\section{Introduction}

\noindent
For a given set $N=\{1,\dots,n\}$
of $n$ voters a simple game is a function $\chi:2^N\rightarrow\{0,1\}$ which is monotone, i.e.\ $\chi(S)\le \chi(T)$
for all $S\subseteq T\subseteq N$, and fulfills $\chi(\emptyset)=0$, $\chi(N)=1$. Here $2^N$ denotes the set of all
subsets of $N$. Those subsets are also called coalitions and $N$ is called the grand coalition. By representing the
subsets of $N$ by their characteristic vectors in $\{0,1\}^n$ we can also speak of a (monotone) Boolean function.
If $\chi(S)=1$ then $S$ is called a winning coalition and otherwise a losing coalition. An important subclass
is the class of weighted voting games for which there are weights $w_i$ for $i\in N$ and a quota $q>0$ such that
the condition $\sum_{i\in S} w_i\ge q$ implies coalition $S$ is winning and the condition
$\sum_{i\in S} w_i<q$ implies coalition $S$ is losing. One attempt to generalize weighted voting games was the
introduction of roughly weighted games, where coalitions $S$ with $\sum_{i\in S} w_i=q$ can be either winning or
losing independently from each other. \footnote{Some authors, e.g.\ \cite{hierarchies}, allow $q=0$, which makes
sense in other contexts like circuits or Boolean algebra.
Later on, we want to rescale the quota $q$ to one, so that we forbid a quota of zero by definition. Another unpleasant
consequence of allowing $q=0$ would be that each simple game on $n$ voters is contained in a roughly weighted game
on $n+1$ voters, i.e., we can add to each given simple game a voter who forms a winning coalition on its own to obtain
a roughly weighted game.} As some games being important both for theory and practice are not even roughly weighted,
\cite{hierarchies} have introduced three hierarchies for simple games to \textit{measure} the \textit{distance} of
a given simple game to the class of (roughly) weighted voting games. In this paper we want to study their third
class $\mathcal{C}_\alpha$, where the tie-breaking point $q$ is extended to the interval $[1,\alpha]$ for an
$\alpha\in\mathbb{R}_{\ge 1}$. Given a game $\chi$, the smallest possible value for $\alpha$ is called
the critical threshold-value $\mu(\chi)$ of $\chi$, see the beginning of Section~\ref{sec_preliminaries}. Let
$\s_\mathcal{S}(n)$ denote the largest critical threshold-value within the class of simple games
$\chi\in\mathcal{S}_n$ on $n$ voters. 
By $Spec_\mathcal{S}(n):=\{\mu(\chi)\mid \chi\in\mathcal{S}_n\}$ we denote the set of possible critical threshold
values.

During the program of classification of simple games, see e.g.\ \cite{1112.91002}, several subclasses have been
proposed and analyzed. Although weighted voting games are one of the most studied and most simple forms of simple
games, they have the shortcomming of not covering all games. The classes $\mathcal{C}_\alpha$ resolve this by
introducing a parameter $\alpha$, so that by varying $\alpha$ the classes of games can be made as large as possible.
The critical threshold value in some sense measures the complexity of a given game. Another such measure is the
dimension of a simple game, see e.g.\ \cite{0765.90030}. Here we observe that there is no direct relation between
these two concepts, i.e.\ simple games with dimension $1$ have a critical threshold value of $1$, but simple games
with dimension larger than $1$ can have arbitrarily large critical threshold values.

Also graphs have been proposed as a suitable representational language for coalitional games. There are a lot of
different graph-based games like e.g.\ shortest path games, connectivity games, minimum cost spanning tree games,
and network flow games. The players of a network flow game are the edges in an edge weighted graph, see
\cite{0773.90097} and \cite{0498.90030}. For so called threshold network flow games, see e.g.\ \cite{pre05940089},
a coalition of edges is winning if and only if those edges allow a flow from a given source to a sink which meets
or exceeds a given quota or threshold. Here the same phenomenon as for weighted voting games arises, i.e.\ those
graph based \textit{weighted} games are not fully expressive, but general network flow games are (within the class
of stable games). Similarly, one can define a hierarchy by requesting a flow of at least $1$ for each winning
coalition and a flow of at least $\alpha$ for each losing coalition.

The concept of the cost of stability was introduced in \cite{Bachrach:2009:CSC:1692490.1692502}. It asks for the
minimum amount of external payment given to the members of a coaltion to ensure stability in a coalition game,
i.e., to guarantee a non-empty core. It will turn out that the cost of stability is closely related to the notion
of $\alpha$-roughly weightedness. For network flow games some results on the cost of stability can be found in
\cite{pre05617303}.

Another line of research, which is related with our considerations, looks at the approximability of Boolean
functions by linear threshold functions, see \cite{approximation_of_linear_threshold}.

\medskip

In \cite{hierarchies} the authors have proven the bounds
$\frac{1}{2}\cdot\left\lfloor\frac{n}{2}\right\rfloor\le \s_\mathcal{S}(n)\le\frac{n-2}{2}$ and
determined the spectrum for $n\le 6$. For odd numbers of voters we slightly improve the lower bound to
$\s_\mathcal{S}(n)\ge \left\lfloor\frac{n^2}{4}\right\rfloor/n$, which is conjectured to be tight. As upper bound
we prove $\s_\mathcal{S}(n)\le\frac{n}{3}$. In order to determine the exact values of $\s_\mathcal{S}(n)$ for small
numbers of voters we provide an integer linear programming formulation. This approach is capable to treat cases
where exhaustive enumeration is computationally infeasible due to the rapidly increasing number of voting structures.
Admittedly, this newly introduced technique, which might be applicable in several other contexts in algorithmic
game theory too, is still limited to a rather small number of voters.

From known results on the spectrum of the determinants of binary $n\times n$-matrices we are able to  conclude some
information on the spectrum of the possible critical threshold values.

The same set of problems can also be studied for subclasses of simple games and we do so for complete simple games,
denoted here by $\mathcal{C}$. Here we conjecture that the maximum critical threshold value $\s_\mathcal{C}(n)$ of
a complete simple game on $n$ voters is bounded by a constant multiplied by $\sqrt{n}$ on both sides. A proof could
be obtained for the lower bound, and, for some special subclasses of complete simple games, also for the upper
bound. In general, we can show that $\s_\mathcal{C}(n)$ grows slower than any linear function reflecting the
valuation that complete simple games are somewhat \textit{nearer} to weighted voting games than general simple games.

The remaining part of this paper is organized as follows:
After this introduction we present the basic definitions and results on linear programs determining the critical
threshold value of a simple game or a complete simple game in Section~\ref{sec_preliminaries}. In
Section~\ref{sec_certificates} we provide certificates for the critical threshold value. General lower and upper
bounds on the maximum possible critical threshold values $\s_\mathcal{S}(n)$ and $\s_\mathcal{C}(n)$ of simple
games and complete simple games are the topic of Section~\ref{sec_maximal_critical_alpha}. In 
Section~\ref{sec_ilp_max_alpha} we provide an integer linear programming formulation to determine the exact value
$\s_\mathcal{S}(n)$ and $\s_\mathcal{C}(n)$. To this end we utilize the dual of the linear program determining
the critical threshold value. In Section~\ref{sec_spectrum} we give some restrictions on the set of possible critical
threshold values and tighten the findings of \cite{hierarchies}. We end with a conclusion in Section~\ref{sec_conclusion}.

\section{Preliminaries}
\label{sec_preliminaries}

\noindent
In this paper we want to study different classes of \textit{voting structures}. As abbreviation for the most
general class we use the notation $\mathcal{B}_n$ for the set of Boolean functions $f:2^N\rightarrow \{0,1\}$
with $f(\emptyset)=0$ on $n$ variables\footnote{We remark that usually $f(\emptyset)=1$ is possible for Boolean
functions too. In our context the notion of $\alpha$-roughly weightedness makes sense for $f(\emptyset)=0$, so
that we generally require this property. Later on, we specialize these sets to monotone Boolean functions with
the additional restriction $f(N)=1$, called simple games, and use the notation $\mathcal{S}_n$. Even more
refined subclasses are the set $\mathcal{C}_n$ of complete simple games and the set $\mathcal{W}_n$ of weighted
voting games on $n$ voters. These sets are ordered as
$\mathcal{B}_n\supseteqq \mathcal{S}_n\supseteqq \mathcal{C}_n\supseteqq\mathcal{W}_n$, where the inclusions
are strict if $n$ is large enough. In order to state examples in a compact manner we often choose weighted
voting games $\chi$, since they can be represented by $[q;w_1,\dots,w_n]$, where $q$ is a quota and the $w_i$
are weights. We have $\chi(S)=1$ if and only if the sum $\sum_{i\in S} w_i\ge q$ for each subset $S\subseteq N$.}
As a shortcut for the sum of weights $\sum_{i\in S} w_i$ of a coalition $S\subseteq N$ we will use $w(S)$ in the
following.

In this section we state the preliminaries, i.e., we define the mentioned classes of voting structures and provide
tailored characterizations of the criticial threshold value within these classes. As a first result we determine
the largest possible critical threshold value for Boolean functions in Lemma~\ref{lemma_s_boolean}. Since it is
closely related, we briefly introduce the concept of the cost of stability for binary 
voting structures.
\begin{definition}
  \label{def_alpha_roughly_weighted}
  A (Boolean) function $f:2^N\rightarrow \{0,1\}$ with $f(\emptyset)=0$ is called \emph{$\alpha$-roughly weighted}
  for an $\alpha\in\mathbb{R}_{\ge 1}$
  if there are weights
  $w_1,\dots,w_n\in\mathbb{R}$ fulfilling
  $$
    w(S)\ge1\quad\quad\forall S\subseteq N:f(S)=1
  $$
  and
  $$
    w(S)\le \alpha\quad\quad\forall S\subseteq N:f(S)=0.
  $$
\end{definition}

We remark that a function $f$ with $f(\emptyset)=1$ cannot be $\alpha$-roughly weighted for any $\alpha\in\mathbb{R}$.
In contrast to most definitions of roughly weighted games we allow negative weights, in the first run, and consider
a wider class than simple games in our initial definition, i.e.\ Boolean functions with $f(\emptyset)=0$.
Later on, we will focus on subclasses of $\mathcal{B}_n$, where we can assume that all weights are non-negative. By
$\ar$ (instead of $\mathcal{C}_\alpha$ as in \cite{hierarchies}) we denote the class of all $\alpha$-roughly weighted
Boolean functions $f$ with $f(\emptyset)=0$. If $f\in\ar$ but $f\notin\mathcal{T}_{\alpha'}$ for all $1\le\alpha'<\alpha$,
we call $\alpha$ the critical threshold value $\mu(f)$ of $f$. Given $f$ we can determine the critical threshold value
using the following linear program: 
\begin{equation}
\label{lp_critical_alpha_value}
\begin{array}{rl}
  \text{Min} & \alpha  \\
   & w(S) \ge 1       \quad\forall S\subseteq N:f(S)=1 \\
   & w(S) \le \alpha  \quad\forall S\subseteq N:f(S)=0 \\
   & \alpha\ge 1\\
   & w_1,\dots,w_n\in\mathbb{R}
\end{array}
\end{equation}
We consider it convenient to explicitly add the constraint $\alpha\ge 1$ in Definition~\ref{def_alpha_roughly_weighted}, 
in accordance with the definition in \cite{hierarchies}, and in the linear program~(\ref{lp_critical_alpha_value}).
Otherwise we would
obtain the optimal solution $\alpha=0$ for the weighted game $[2;1,1]\in\mathcal{B}_2$ or the optimal solution
$\alpha=\frac{2}{3}$ for the weighted game $[3;2,2,1,1]\in\mathcal{B}_4$ using the weights $w_1=w_2=\frac{2}{3}$
and $w_3=w_4=\frac{1}{3}$. Since there are no coalitions with weights strictly between $\frac{2}{3}$ and $1$ there
are no contradicting implications. Arguably, values less than $1$ contain more information, but on the other hand
makes notation more complicated. To avoid any misconception we directly require $\alpha\ge 1$ (as in
Definition~\ref{def_alpha_roughly_weighted}) to guarantee non-contradicting implications independently from the
possible weights of the coalitions.

At first, we remark that the inequality system~(\ref{lp_critical_alpha_value}) has at least one feasible solution
given by $w_i=1$ for all $1\le i\le n$ and $\alpha=n$. Next we observe that the critical threshold value is a rational
number, as it is the optimum solution of a linear programming problem with rational coefficients, and that we can
restrict ourselves to rational weights $w_i$. For a general Boolean function $f:2^N\rightarrow\{0,1\}$ with
$f(\emptyset)=0$ negative weights may be necessary to achieve the critical threshold value. An example is given by
the function $f$ of three variables whose entire set of coalitions $S$ with $f(S)=1$ is given by
$\{\{1\},\{2\},\{1,2\}\}$. By considering the weights $w_1=w_2=1$, $w_3=-2$ we see that it is $1$-roughly weighted.
On the other hand we have the inequalities $w_1\ge 1$, $w_2\ge 1$, and $w_1+w_2+w_3\le\alpha=1$ from which we conclude
$w_3\le -1$. Another way to look at this example is to say that the critical threshold value would be $2$ if only
non-negative weights are allowed. (Here $n=3$ voters are the smallest possibility, i.e.\ for $n\le 2$ there are
non-negative realizations for the critical threshold value.)

A quite natural question is to ask for the largest critical threshold value $\mu(f)$ within the class of all
Boolean functions $f:2^N\rightarrow \{0,1\}$ with $f(\emptyset)=0$, which we denote by $\s_\mathcal{B}(n)$,
i.e.\ $\s_\mathcal{B}(n)=\max\{ \mu(f)\mid f\in \mathcal{B}_n\}$.

\begin{lemma}
  \label{lemma_s_boolean}
   $\s_\mathcal{B}(n)=n$.
\end{lemma}
\begin{proof}
  By choosing the weights $w_i=1$ for all $1\le i\le n$ we have $1\le w(S)\le n$ for all $\emptyset\neq S\subseteq N$.
  Thus all functions  $f:2^N\rightarrow \{0,1\}$ with $f(\emptyset)=0$ are $n$-roughly weighted. The maximum
  $\s_\mathcal{B}(n)=n$ is attained for example at the function with $f(N)=0$ and $f(\{i\})=1$ for all $1\le i\le n$.
  Since the singletons $\{i\}$ are winning, we have $w_i\ge 1$ for all $i\in N$, so that $w(N)\ge n$ while $N$ is a
  losing coalition.
\end{proof}

We would like to remark that if we additionally require $f(N)=1$, then the critical threshold value is at most $n-1$,
which is tight (the proof of Lemma~\ref{lemma_s_boolean} can be easily adapted). 

\bigskip
\bigskip

More interesting subclasses of Boolean functions with $f(\emptyset)=0$ are simple games, i.e.\ monotone Boolean
functions with $f(\emptyset)=0$ and $f(N)=1$, where $f(S)\le f(T)$ for all $S\subseteq T$. By $\ar\cap\mathcal{S}_n$
we denote the class of all $\alpha$-roughly weighted simple games consisting of $n$ voters and by
$\s_\mathcal{S}(n):=\max\{\mu(f)\mid f\in\mathcal{S}_n\}$ the largest critical threshold value within the class
of simple games consisting of $n$~voters. For simple games we can restrict ourselves to non-negative weights and
can drop some of the inequalities in the linear program~(\ref{lp_critical_alpha_value}). (This is not true
for general Boolean functions as demonstrated in the previous example.)

\begin{lemma}
  \label{lemma_non_negative_weights_simple_game}
  All simple games $\chi\in \ar\cap\mathcal{S}_n$ admit a representation in non-negative weights.
\end{lemma}
\begin{proof}
  Let $w_i\in\mathbb{R}$, for $1\le i\le n$, be suitable weights. We set $w_i':=\max(w_i,0)\in\mathbb{R}_{\ge 0}$
  for all $1\le i\le n$. For each winning coalition $S\subseteq N$ we have $w'(S)\ge w(S)\ge 1$. Due to the
  monotonicity property of simple games for each losing coalition $T\subseteq N$ the coalition
  $T':=\{i\in T\,:\,w_i\ge 0\}$ is also losing. 
  Thus we have $w'(T)\le w(T')\le \alpha$.
\end{proof}

We remark that we have not used $\chi(\emptyset)=0$ or $\chi(N)=1$ so that the statement can be slightly generalized.

\begin{definition}
  Given a simple game $\chi$ a coalition $S\subseteq N$ is called a \emph{minimal winning coalition} if $\chi(S)=1$
  and $\chi(S')=0$ for all proper subsets $S'$ of $S$.  Similarly, a coalition $T\subseteq N$ is called a \emph{maximal
  losing coalition} if $\chi(T)=0$ and $\chi(T')=1$ for all $T'\subseteq N$ where $T$ is a proper subset of $T'$. By
  $\mathcal{W}$ we denote the set of minimal winning coalitions and by $\mathcal{L}$ the set of maximal losing coalitions.
\end{definition}

We would like to remark that a simple game can be completely reconstructed from either the set $\mathcal{W}$ of its
minimal winning coalitions or the set $\mathcal{L}$ of its maximal losing coalitions, i.e.\ a coalition $S\subseteq N$
is winning if and only if it contains a subset $S'\in \mathcal{W}$. Similarly, a coalition $T\subseteq N$
is losing if there is a $T'\in\mathcal{L}$ with $T\subseteq T'$. 

\begin{proposition}
  \label{prop_lp_critical_alpha_simple_game}
  The critical threshold value $\mu(\chi)$ of a simple game $\chi\in\mathcal{S}_n$ is given by the optimal
  target value of the following linear program:
  $$
    \begin{array}{rl}
      \text{Min} & \alpha  \\
       & w(S) \ge 1       \quad\forall S\in\mathcal{W} \\
       & w(S) \le \alpha  \quad\forall S\in\mathcal{L} \\
       & \alpha\ge 1\\
       & w_1,\dots,w_n\ge 0
    \end{array}
  $$
\end{proposition}
\begin{proof}
  Due to Lemma~\ref{lemma_non_negative_weights_simple_game} we can assume w.l.o.g.\ that $w_1,\dots,w_n\ge 0$. With this
  it suffices to prove that a feasible solution of the stated linear program is also feasible for the linear
  program~(\ref{lp_critical_alpha_value}). Let $S\subseteq N$ be an arbitrary winning coalition, i.e., $\chi(S)=1$.
  Since there exists an $S'\in\mathcal{W}$ with $S'\subseteq S$ we have $$w(S)\overset{w_i\ge 0}{\ge} w(S')\ge 1.$$ 
  Similarly, for each losing coalition $T\subseteq N$ there exists a $T'\in\mathcal{L}$ with
  $T\subseteq T'$ so that we have
  $$
    w(T)\overset{w_i\ge 0}{\le} w(T')\le \alpha.
  $$
\end{proof}

Again, we have not used $\chi(\emptyset)=0$ or $\chi(N)=1$ in the proof.

\bigskip
\bigskip

A well studied subclass of simple games (and superclass of weighted voting games) arises from Isbell's desirability
relation, see \cite{0083.14301}: We write $i\sqsupset j$ for two voters $i,j\in N$ iff we have
$\chi\Big(\{i\}\cup S\backslash\{j\}\Big)\ge\chi (S)$ for all $j\in S\subseteq N\backslash\{i\}$. A pair $(N,\chi)$
is called a \emph{complete simple game} if it is a simple game and the binary relation $\sqsupset$ is a total preorder.
To factor out symmetry we assume $i\sqsupset j$ for all $1\le i<j\le n$, i.e.\ voter~$i$ is at least as powerful as
voter~$j$, in the following. We abbreviate $i\sqsupset j$, $j\sqsupset i$ by $i~\square~j$ forming equivalence classes
of voters $N_1,\dots,N_t$. Let us denote $|N_i|=n_i$ for $1\le i\le t$. We assume that those equivalence
classes are ordered with decreasing influence, i.e.\ for $u\le v$, $i\in N_u$, $j\in N_v$ we have $i\sqsupset j$.
A coalition in a complete simple game can be described by the numbers $a_h$ of voters from equivalence class $N_h$,
i.e.\ by a vector $(a_1,\dots,a_t)$. Note that the same vector represents ${n_1\choose a_1}{n_2\choose a_2}\dots {n_t\choose a_t}$
coalitions that only differ in equivalent voters.

To transfer the concept of minimal winning coalitions and maximal losing coalitions to vectors, we need a suitable
partial ordering:
\begin{definition}
  \label{def_smaller_vector}
  For two integer vectors $\widetilde{a}=(a_1,\dots,a_t)$ and
  $\widetilde{b}=(b_1,\dots,b_t)$ we write $\widetilde{a}\preceq \widetilde{b}$ if
  we have
  $
    \sum\limits_{i=1}^{k} a_i \le \sum\limits_{i=1}^{k} b_i
  $
  for all $1\le k\le t$. For $\widetilde{a}\preceq \widetilde{b}$ and $\widetilde{a}\neq \widetilde{b}$ we use
  $\widetilde{a}\prec\widetilde{b}$ as an abbreviation. If neither $\widetilde{a}\preceq \widetilde{b}$ nor
  $\widetilde{b}\preceq \widetilde{a}$ holds we write $\widetilde{a}\bowtie \widetilde{b}$.
\end{definition}
In words, we say that $\widetilde{a}$ is smaller than $\widetilde{b}$ if $\widetilde{a}\prec\widetilde{b}$ and
that $\widetilde{a}$ and $\widetilde{b}$ are incomparable if $\widetilde{a}\bowtie \widetilde{b}$. 

With Definition \ref{def_smaller_vector} and the representation of coalitions as vectors in $\mathbb{N}^t$ at
hand, we can define: 
\begin{definition}
\label{def_shift_minimal}
A vector $\widetilde{m}:=(m_1,\dots,m_t)$ in a complete simple game\\ 
$\Big((n_1,\dots,n_t),\chi\Big)$ is a \emph{shift-minimal winning vector} if $\widetilde{m}$ is a
winning vector and every vector $\widetilde{m}'\prec\widetilde{m}$ is
losing. Analogously, a vector $\widetilde{m}$ is a \emph{shift-maximal losing vector} if $\widetilde{m}$
is a losing vector and every vector $\widetilde{m}'\succ\widetilde{m}$ is winning.
\end{definition}

As an example we consider the complete simple game $\chi\in\mathcal{C}_4$ whose minimal winning
coalitions are given by $\{1,2\}$, $\{1,3\}$, $\{1,4\}$, and $\{2,3,4\}$. The equivalence classes of voters
are given by $N_1=\{1\}$ and $N_2=\{2,3,4\}$. With this the shift-minimal winning vectors are given by $(1,1)$
and $(0,3)$. By $\overline{\mathcal{W}}$ we denote the set of shift-minimal winning vectors and by
$\overline{\mathcal{L}}$ the set of shift-maximal losing vectors. Each complete simple game can be entirely
reconstructed from either $\overline{\mathcal{W}}$ or $\overline{\mathcal{L}}$. 

In \cite{complete_simple_games} there is a very useful parameterization theorem for complete simple games:
\begin{theorem}
  \label{thm_csg_characterization}

  \vspace*{0mm}

  \noindent
  \begin{itemize}
   \item[(a)] Let a vector $$\widetilde{n}=(n_1,\dots,n_t)\in\mathbb{N}_{>0}^t$$ and a matrix
              $$\mathcal{M}=\begin{pmatrix}m_{1,1}&m_{1,2}&\dots&m_{1,t}\\m_{2,1}&m_{2,2}&\dots&m_{2,t}\\
              \vdots&\vdots&\ddots&\vdots\\m_{r,1}&m_{r,2}&\dots&m_{r,t}\end{pmatrix}=
              \begin{pmatrix}\widetilde{m}_1\\\widetilde{m}_2\\\vdots\\\widetilde{m}_r\end{pmatrix}$$
              be given, which satisfies the following properties:
              \begin{itemize}
               \item[(i)]   $0\le m_{i,j}\le n_j$, $m_{i,j}\in\mathbb{N}_{\ge 0}$ for $1\le i\le r$, $1\le j\le t$,
               \item[(ii)]  $\widetilde{m}_i\bowtie\widetilde{m}_j$ for all $1\le i<j\le r$,
               \item[(iii)] for each $1\le j<t$ there is at least one row-index $i$ such that
                            $m_{i,j}>0$, $m_{i,j+1}<n_{j+1}$ if $t>1$ and $m_{1,1}>0$ if $t=1$, and
               \item[(iv)]  $\widetilde{m}_i\gtrdot \widetilde{m}_{i+1}$ for $1\le i<t$ (lexicographic order).
              \end{itemize}
              Then there exists a complete simple game $(N,\chi)$ whose equivalence classes of voters
              have cardinalities as in $\widetilde{n}$ and whose shift-minimal winning vectors coincide with the
              rows of $\mathcal{M}$.
   \item[(b)] Two complete games $\left(\widetilde{n}_1,\mathcal{M}_1\right)$ and $\left(\widetilde{n}_2,\mathcal{M}_2\right)$
              are isomorphic, i.e., there exists a permutation of the voters so that the games are equal, if and only
              if $\widetilde{n}_1=\widetilde{n}_2$ and $\mathcal{M}_1=\mathcal{M}_2$.
  \end{itemize}
\end{theorem}
The rows of $\mathcal{M}$ correspond to the shift-minimal winning vectors whose number is denoted by $r$. The
number of equivalence classes of voters is denoted by $t$.

By $\s_\mathcal{C}(n):=\{\max \mu(\chi)\mid \chi\in\mathcal{C}_n\}$ we denote the largest critical threshold value 
within the class of complete simple games on $n$ voters. As $\overline{\mathcal{W}}\subseteq \mathcal{W}$ and
$\overline{\mathcal{L}}\subseteq\mathcal{L}$ we want to provide a linear programming formulation for the critical
threshold value $\mu(\chi)$ of a complete simple game $\chi\in\mathcal{C}_n$, similar to
Proposition~\ref{prop_lp_critical_alpha_simple_game}, based on shift-minimal winning and shift-maximal losing vectors.
At first, we show that we can further restrict the set of weights. To this end we call a feasible solution $w$ of the
inequality system in Proposition~\ref{prop_lp_critical_alpha_simple_game}, where $\alpha$ is given, a 
\emph{representation} (with respect to $\alpha$).

\begin{lemma}
  All complete simple games $\chi\in \ar\cap\mathcal{C}_n$ admit a representation with weights satisfying
  $w_1\ge\dots\ge w_n\ge 0$.
\end{lemma}
\begin{proof}
  As $\chi\in\mathcal{C}_n\subseteq\mathcal{S}_n$ is a simple game, there exists a representation with
  weights $w_1',\dots,w_n'\in\mathbb{R}_{\ge 0}$ due to Lemma~\ref{lemma_non_negative_weights_simple_game}. Let
  $(j,h)$ be the lexicographically smallest pair such that $w_j'<w_h'$ and $j<h$. By $\tau$ we denote the
  transposition $(j,h)$, i.e.\ the permutation that swaps $j$ and $h$, and set $w_i:=w_{\tau(i)}'$.
  
  For a winning coalition $S$ with $j\in S$, $h\notin S$ we have $w(S)\ge w'(S)\ge 1$. If $S$ is a winning
  coalition with $j\notin S$, $h\in S$ then $\tau(S)$ is a winning coalition too and we have $w(S)=w'(\tau(S))\ge 1$.
  For a losing coalition $T$ with $j\notin T$, $h\in T$ we have $w(T)\le w'(T)\le\alpha$. If $T$ is a losing
  coalition with $j\in T$, $h\notin T$ then $\tau(T)$ is a losing coalition too and we have $w(T)=w'(\tau(T))\le \alpha$.
  
  By recursively applying this argument we can construct representing weights fulfilling
  $w_1\ge\dots\ge w_n\ge 0$.
\end{proof}

We remark that the previous complete simple game with minimal winning coalitions $\{1,2\}$, $\{1,3\}$, $\{1,4\}$,
and $\{2,3,4\}$ can be represented as a weighted voting game $[4;3,2,1,1]$. Another representation of the same game
using equal weights for equivalent voters would be $[3;2,1,1,1]$.

\begin{lemma}
  \label{lemma_weights_csg}
  All complete simple games $\chi\in \ar\cap\mathcal{C}_n$ admit a representation with weights  $w_1\ge\dots\ge w_n\ge 0$
  where voters of the same equivalence class have the same weight.
\end{lemma}
\begin{proof}
  Let $w_1'\ge\dots\ge w_n'\ge 0$ be a representation of $\chi$ and $N_1,\dots,N_t$ the set of equivalence
  classes of voters. By $1\le j\le t$ we denote the smallest index such that not all voters in $N_j$ have the same weight
  and define new weights $w_i:=w_i'$ for all $i\in N\backslash N_j$ and $w_i:=\frac{\sum\limits_{h\in N_j} w_h'}{|N_j|}$,
  i.e.\ the arithmetic mean of the previous weights in $N_j$. By recursively applying this construction  we obtain a
  representation with the desired properties. It remains to show that the new weights $w_i$ fulfill the $\alpha$-conditions.
  
  Let $S$ be a winning coalition with $k=\left|S\cap N_j\right|$. By $S'$ we denote the union of $S\backslash N_j$ and
  the $k$ lightest voters from $N_j$. Since $S'$ is a winning coalition too we have $w(S)\ge w'(S')\ge 1$. Similarly,
  let $T$ be a losing coalition with $k=\left|T\cap N_j\right|$: By $T'$ we denote the union of $T\backslash N_j$ and
  the $k$ heaviest voters from $N_j$. Since $T'$ is also a losing coalition we have $w(T)\le w'(T')\le\alpha$.
\end{proof}

\begin{lemma}
  \label{lemma_lp_critical_alpha_complete_simple_game}
  The critical threshold value $\mu(\chi)$ of a complete simple game $\chi\in\mathcal{C}_n$
  with $t$ equivalence classes of voters is given by the optimal target value of the following linear program:
  $$
    \begin{array}{rl}
      \text{Min} & \alpha  \\
       & \sum\limits_{i=1}^t a_iw_i \ge 1       \quad\forall (a_1,\cdots,a_t)\in\overline{\mathcal{W}} \\
       & \sum\limits_{i=1}^t a_iw_i \le \alpha  \quad\forall (a_1,\cdots,a_t)\in\overline{\mathcal{L}} \\
       & \alpha\ge 1\\
       & w_i\ge w_{i+1}\quad\forall 1\le i\le t-1\\
       & w_t\ge 0
    \end{array}
  $$
\end{lemma}
\begin{proof}
  Due to Lemma~\ref{lemma_weights_csg} we can assume that for the critical threshold value $\mu(\chi)=\alpha$
  there exists a feasible weighting fulfilling the conditions of the stated linear program. It remains to show
  that $w(W)\ge 1$ and $w(L)\le \alpha$ holds for all shift-winning vectors $W$ and all losing vectors $L$. Therefore,
  we denote by $W'\in\overline{\mathcal{W}}$ an arbitrary shift-minimal winning vector with $W\succeq W'$
  and by $L'\in\overline{\mathcal{L}}$ an arbitrary shift-maximal losing vector with $L\preceq L'$. The proof
  is finished by checking $w(L)\le w(L')\le \alpha$ and $w(W)\ge w(W')\ge 1$.
\end{proof}

So, for complete simple games the number of constraints could be further reduced. In this context we remark
that by additionally disregarding the conditions $w_i\ge w_{i+1}$ from the linear program we would lose the
information about the order on equivalence classes. This effect is demonstrated  by the following example.
Let us consider the complete simple game $(n_1,n_2)=(15,4)$ with unique shift-minimal winning vector $(7,2)$.
There are two shift-maximal losing vectors: $(8,0)$ and $(6,4)$. Choosing the special solution
$w_1=\frac{1}{14}$, $w_2=\frac{1}{4}$, $\alpha=\frac{3}{2}$ would be feasible for
\begin{eqnarray*}
  7w_1+2w_2\ge 1\\
  8w_1\le \alpha\\
  6w_1+4w_2\le \alpha\\
  \alpha\ge 1\\
  w_1,w_2\ge 0
\end{eqnarray*}
For the coalition $(8,1)$ we obtain the weight $8w_1+1w_2=\frac{23}{28}<1$, so that it should be a losing coalition,
which is a contradiction to $(8,1)\succeq(7,2)$. So we have to use the ordering on the weights. 

At the beginning of this section we have argued that the condition $\alpha\ge 1$ is necessary, since
otherwise the optimal target value of the stated linear programming formulations will not coincide with
$\mu(\chi)$ in all cases. On the other hand, if $z^\star(\chi)$ denotes the optimal target value of one of
the stated LPs, where we have dropped the condition $\alpha\ge 1$, then we have
$$
  \mu(\chi)=\max\!\left(z^\star(\chi),1\right).
$$
In the following we will drop the condition $\alpha\ge 1$ whenever it seems beneficial for the ease of a shorter
presentation while having the just mentioned exact correspondence in mind.


\bigskip

An important solution concept in cooperative game theory is the \emph{core}, i.e.\ the set of all stable imputations,
see e.g.\ \cite{1229.91004} for an introduction. Since the core can be empty under certain circumstances, the possibility
of external payments was considered in order to stabilize the outcome, see \cite{Bachrach:2009:CSC:1692490.1692502}. The
external party quite naturally is interested in minimizing its expenditures. This leads to the concept of the
\emph{cost of stability} ($CoS$) of a coalition game. Skipping the relation of $CoS$ with the core, we directly define
the cost of stability $CoS(f)$ of a given Boolean function $f$ with $f(\emptyset)=0$ as the solution of the following
linear program:
\begin{eqnarray}
  \text{Min} && \Delta  \\
  \Delta &\ge& 0\\
  \sum_{i\in N} p_i &=& f(N)+\Delta\label{ie_cos_normalization}\\
  \sum_{i\in S} p_i &\ge& f(S)\quad\forall S\subseteq N\label{ie_cos_lower_bound_coalition}\\
  p_i&\ge&0\,\,\,\,\quad\quad\forall i\in N\label{ie_cos_positive}.
\end{eqnarray}
The cost of stability is an upper bound for the critical threshold value:
\begin{lemma}
  \label{lemma_cos}
  For a Boolean function $f\in\mathcal{B}_n$ with $f(N)=1$ we have
  $\mu(f)\le 1+CoS(f)$.
\end{lemma}
\begin{proof}
   Let $p_1,\dots, p_n$, $\Delta$ be an optimal solution for the above linear program for the cost of stability. If
   we choose the weights as $w_i=p_i$, then we have $w_i\in\mathbb{R}$ and we have $w(S)\ge 1$ for all winning
   coalitions $S$ due to constraint~(\ref{ie_cos_lower_bound_coalition}). Applying constraint~(\ref{ie_cos_positive})
   and constraint~(\ref{ie_cos_normalization}) yields $$w(S)=\sum_{i\in S}p_i\le \sum_{i\in N} p_i=f(N)+\Delta=1+CoS(f)$$
   for all coalitions $S\subseteq N$. Thus every losing coalition has a weight of at most $1+CoS(f)$.
\end{proof}

Due to $CoS(f)\le n\cdot\max_{S\subseteq N} f(S)\le n$, see Theorem 3.4 in \cite{Bachrach:2009:CSC:1692490.1692502}, we
have $CoS(f)\le n$ for all $f\in\mathcal{B}_n$, where equality is attained for the Boolean function with $f(S)=1$ for
all $S\neq\emptyset$. With respect to Lemma~\ref{lemma_s_boolean} we mention the relation
$$
  c_\mathcal{B}(n)=\max_{f\in\mathcal{B}_n} \mu(f)=\max_{f\in\mathcal{B}_n} CoS(f)=n.
$$
On the other hand, we observe that the ratio between $CoS(f)$ and $\mu(f)$ can be quite large. Theorem 3.3 in
\cite{Bachrach:2009:CSC:1692490.1692502} states $CoS(\chi)=\frac{n}{\lceil q\rceil} -1$ for the weighted voting game
$\chi=[q;w,\dots,w]$, while we have $\mu(\chi)=1$. Setting $w=q=1$ we see that the ration can become at least as
large as $n-1$.

By imposing more structure on the set of feasible games, the bound $CoS(f)\le n$, for $f\in\mathcal{B}_n$, could
be reduced significantly. To this end we introduce further notation:
\begin{definition}
  \label{def_super_additive}
  A Boolean function $f\in\mathcal{B}_n$ is called \emph{super-additive} if we have $f(S)+f(T)\le f(S\cup T)$ for
  all disjoint coalitions $S,T\subseteq N$. It is called \emph{anonymous} if we have $f(S)=f(T)$ for all coalitions
  $S,T\subseteq N$ with $|S|=|T|$, i.e.\ the outcome only depends on the cardinality of the coalition.
\end{definition}
In our context super-additivity means that each pair of winning coalitions has a non-empty intersection, which is
also called a \emph{proper} game. These are the most used voting games for real world institutions.

\section{Certificates}
\label{sec_certificates}

\noindent
In computer science, more precisely in complexity theory, a certificate is a string that certifies the answer to
a membership question (or the optimality of a computed solution). In our context we e.g.\ want to know whether a
given simple game $\chi\in\mathcal{S}_n$ is $\alpha$-roughly weighted. If the answer is yes, we just need to state
suitable weights. Given the weights, the answer then can be checked by testing the validity of the inequalities in
the linear program of Proposition~\ref{prop_lp_critical_alpha_simple_game}. Since both $\mathcal{W}$ and $\mathcal{L}$
form antichains, i.e.\ no element is contained in another, we can conclude from Sperner's theorem that at most
$2{{n} \choose {\lfloor n/2\rfloor}}+n+1$ inequalities have to be checked. But also in the other case, where the
answer is no, we would like to have a computational witness that $\chi$ is not $\alpha$-roughly weighted.

For weighted voting games trading transforms, see e.g.\ \cite{0943.91005}, can serve as a certificate for
non-weightedness. In \cite{Gvozdeva201120} this concept has been transfered to roughly weighted games and it was
proven that for each non-weighted simple game consisting of $n$ voters there exists a trading transform
of length at most $\left\lfloor(n+1)\cdot 2^{\frac{1}{2}n\log_2 n}\right\rfloor$.

Using the concept of duality in linear programming one can easily construct a certificate for the fact that a given 
voting structure $\chi$ is not $\alpha'$-roughly weighed for all $\alpha'<\alpha$, where $\alpha\ge 1 $ is fixed.
To be more precise, we present a certificate for the inequality $\mu(\chi)\ge\alpha$.  

The dual of a general linear program $\min c^Tx, Ax\ge b, x\ge 0$ (called primal) is given by $\max b^Ty, A^Ty\le c, y\ge 0$.
The strong duality theorem, see e.g.\ \cite{vanderbei}, states that if the primal has an optimal solution, $x^\star$,
then the dual also has an optimal solution, $y^\star$, such that $c^Tx^\star=b^Ty^\star$. As mentioned before, the
linear program for the determination of the critical threshold value always has an optimal solution, so that we can
apply the strong duality theorem to obtain a certificate.

Considering only a subset of the winning coalitions for the determination of the critical threshold value means removing some
constraints of the corresponding linear program. This enlarges the feasible set such that the optimal solution will eventually
decrease but not increase. For further utilization we state the resulting lower bound for the critical threshold value of this
approach:
\begin{lemma}
  \label{lemma_lower_bound_critical_alpha_value_simple_games}
  For a given simple game $\chi\in\mathcal{S}_n$ let $W'$ be a subset of its winning coalitions and $L'$ be a subset of
  its losing coalitions. If $(u,v)$ is a feasible solution of the following linear program with target value $\alpha'$ then
  we have $\mu(\chi)\ge\alpha'$.
  $$
    \begin{array}{rl}
      \text{Max} & \sum\limits_{S\in W'} u_S \\
      & \sum\limits_{S\in W':i\in S} u_S\,-\,\sum\limits_{T\in L':i\in T} v_T \le 0\quad\forall 1\le i\le n\\
      & \sum\limits_{T\in L'} v_T \le 1\\
      & u_S\ge 0\quad\forall S\in W'\\
      & v_T\ge 0\quad\forall T\in L'
    \end{array}
  $$
\end{lemma}
\begin{proof}
  The stated linear program is the dual of
  $$
    \begin{array}{rl}
      \text{Min} & \alpha \\
      & \sum\limits_{i\in S} w_i\ge 1\quad \forall S\in W'\\
      & \alpha-\sum\limits_{i\in T} w_i \ge 0\quad\forall T\in L'\\
    \end{array}  
  $$
  $$
    \begin{array}{rl}
      & w_i\ge 0\quad \forall 1\le i\le n,
    \end{array}
  $$
  which is a relaxation of the linear program~(\ref{lp_critical_alpha_value}) determining the critical
  threshold value.
\end{proof}

To briefly motivate the underlying ideas we consider an example. Let the simple game~$\chi$ for $5$~voters be defined 
by its set
$\Big\{\{1,2\},\{2,4\},\{3,4\},\{2,5\},$ $\{3,5\}\Big\}$ of minimal winning coalitions. The set of maximal losing
coalitions is given by
$\Big\{\{1,3\},\{2,3\},\{1,4,5\}\Big\}$.  For this example the linear program of
Proposition~\ref{prop_lp_critical_alpha_simple_game} to determine the critical $\alpha$ (after some easy
equivalence transformations) reads as
$$
\begin{array}{rl}
  \text{Min} & \alpha\quad\quad\text{s.t.}  \\
   & w_1+w_2 \ge 1\\
   & w_2+w_4 \ge 1\\
   & w_3+w_4 \ge 1\\
   & w_2+w_5 \ge 1\\
   & w_3+w_5 \ge 1\\
   & \alpha-w_1-w_3 \ge 0\\
   & \alpha-w_2-w_3 \ge 0\\
   & \alpha-w_1-w_4-w_5 \ge 0
\end{array}
$$
$$
\begin{array}{rl}  
   & \alpha\ge 1\\
   & w_1 \ge 0, \dots, w_5 \ge 0
\end{array}
$$
(We have replaced the conditions $w(S)\le \alpha$ for the losing coalitions $S$ by $\alpha-w(S)\ge 0$.)

Running a linear program solver yields the optimal solution $w_1=w_4=w_5=\frac{2}{5}$, $w_2=w_3=\frac{3}{5}$,
and $\alpha=\frac{6}{5}$. By inserting these values into the inequalities of the stated linear program we can check that 
$\chi\in\mathcal{T}_{\frac{6}{5}}\cap\mathcal{S}_5$. Thus the weights form a certificate for this fact.


To obtain a certificate for the fact that $\chi\notin\mathcal{T}_{\alpha'}$ for all
$\alpha'<\frac{6}{5}$, i.e.\ $\mu(\chi)\ge\frac{6}{5}$, we consider the dual problem:
$$
\begin{array}{rl}
  \text{Max} & y_1+y_2+y_3+y_4+y_5+z\quad\quad\text{s.t.}  \\
   & y_1-y_6-y_8 \le 0 \\
   & y_1+y_2+y_4-y_7 \le 0\\
   & y_3+y_5-y_7 \le 0\\
   & y_2+y_3-y_8 \le 0\\
   & y_4+y_5-y_8 \le 0\\
   & y_6+y_7+y_8+z \le 1 \\
   & y_1 \ge 0, \dots, y_8,z \ge 0
\end{array}
$$
An optimal solution is given by $y_1=y_5=y_8=\frac{2}{5}$, $y_2=y_3=\frac{1}{5}$, $y_7=\frac{3}{5}$, and $y_4=y_6=z=0$
with target value $\frac{6}{5}$ (as expected using the strong duality theorem). In combination with the weak duality
theorem, see e.g.\ \cite{vanderbei}, the stated feasible dual solution $(y,z)$  is a certificate for the fact that the
critical threshold value for the simple game $\chi$ is larger or equal to $\frac{6}{5}$. In general, the optimal
solution vector $(y,z)$ has at most $n+1$ non-zero entries so that we obtain a very short certificate.

We would like to remark that one can use the values of the dual variables as multipliers for the inequalities in
the primal problem to obtain the desired bound on the critical threshold value. In our case multiplying all
inequalities with the respective values yields
\begin{eqnarray*}
  &&
  \frac{2}{5}\cdot\left(w_1+w_2\right)
  +\frac{1}{5}\cdot\left(w_2+w_4\right)
  +\frac{1}{5}\cdot\left(w_3+w_4\right)
  +0\cdot\left(w_2+w_5\right)
  +\frac{2}{5}\cdot\left(w_3+w_5\right)\\
  &&
  +0\cdot\left(\alpha-w_1-w_3\right)
  +\frac{3}{5}\cdot\left(\alpha-w_2-w_3\right)
  +\frac{2}{5}\cdot\left(\alpha-w_1-w_4-w_5\right)
  +0\cdot \alpha\\
  &&
  \ge \frac{2}{5}+\frac{1}{5}+\frac{1}{5}+0+\frac{2}{5}+0=\frac{6}{5}
\end{eqnarray*}
which is equivalent to $\alpha\ge \frac{6}{5}$, i.e.\ a certificate for the fact that
$\chi\notin\mathcal{T}_{\alpha'}\cap\mathcal{S}_5$ for $\alpha'<\frac{6}{5}$.

\section{Maximal critical threshold values}
\label{sec_maximal_critical_alpha}

\noindent
In Lemma~\ref{lemma_s_boolean} we have shown that the maximum critical threshold value of a Boolean function
$f:2^N\rightarrow\{0,1\}$ with $f(\emptyset)=0$ is given by $\s_\mathcal{B}(n)=n$. If additionally $f(N)=1$ is
required the upper bound drops to $n-1$ (which is tight). In this section, we want to provide bounds for the
maximal critical threshold values for simple games and complete simple games on $n$ voters. By considering a
complete simple game with two types of voters we can derive a lower bound of $\Omega(\sqrt{n})$ for
$c_\mathcal{C}(n)$. Apart from constants, this bound is conjectured to be tight. This will be substantiated
by upper bounds of $O(\sqrt{n})$ for $c_\mathcal{C}(n)$ for several special subclasses of complete simple games.
For the general case, we can only obtain the result that $c_\mathcal{C}(n)$ is asymptotically smaller than $O(n)$,
which is the asymptotic of the maximum critical threshold value for simple games. Finally, we relate the more
sophisticated upper bounds on the cost of stability from \cite{Bachrach:2009:CSC:1692490.1692502} to upper bounds
for the critical threshold value for other special subclasses of Boolean games. 

\bigskip

The authors of \cite{hierarchies} have proven the bounds $\frac{1}{2}\left\lfloor\frac{n}{2}
\right\rfloor\le \s_\mathcal{S}(n)\le \frac{n-2}{2}$ for $n\ge 4$ and determined the exact values
$\s_\mathcal{S}(1)=\s_\mathcal{S}(2)=\s_\mathcal{S}(3)=\s_\mathcal{S}(4)=1$, $\s_\mathcal{S}(5)=\frac{6}{5}$,
$\s_\mathcal{S}(6)=\frac{3}{2}$. By considering null voters we conclude
$\s_\mathcal{S}(n)\le \s_\mathcal{S}(n+1)$ and $\s_\mathcal{C}(n)\le \s_\mathcal{C}(n+1)$ for all $n\in\mathbb{N}$.

\begin{proposition}
  \label{prop_c_largest_lower_bound}
  For $n\ge 4$ we have $\s_\mathcal{S}(n)\ge\frac{\left\lfloor\frac{n^2}{4}\right\rfloor}{n}$.
\end{proposition}
\begin{proof}
  For the even integers we took an example from \cite{hierarchies} and consider for $n=2k$ the simple game uniquely
  defined by the minimal winning coalitions $W_i=\{2i-1,2i\}$ for $1\le i\le k$. Then the two coalitions
  $L_1=\{1,3,\dots,2k-1\}$ and $L_2=\{2,4,\dots,2k\}$ are maximal losing coalitions. Our example given above is of
  this type ($k=4$). We apply Lemma~\ref{lemma_lower_bound_critical_alpha_value_simple_games} with
  $u_{W_1}=\dots=u_{W_k}=v_{L_1}=v_{L_2}=\frac{1}{2}$ to
  deduce $\s_\mathcal{S}(n)\ge\sum\limits_{i=1}^k \frac{1}{2}=\frac{n}{4}$.
  Using a null voter, as done in \cite{hierarchies}, gives $\s_\mathcal{S}(n)\ge\frac{n-1}{4}$ for odd $n$, where
  $\frac{\left\lfloor\frac{n^2}{4}\right\rfloor}{n}-\frac{n-1}{4}=\frac{n-1}{4n}$.
  
  For odd $n=2k+1$ we consider the simple game uniquely defined by the minimal winning coalitions $W_i=\{i,i+1\}$ for
  $1\le i\le n-1$. Two maximal losing coalitions are given by $L_1=\{1,3,\dots,2k+1\}$ and $L_2=\{2,4,\dots,2k\}$. Next we
  apply Lemma~\ref{lemma_lower_bound_critical_alpha_value_simple_games} and construct a certificate for 
  $\s_\mathcal{S}(n)\ge\frac{(n-1)(n+1)}{4n}=\frac{\left\lfloor\frac{n^2}{4}\right\rfloor}{n}$. We set
  $u_{W_{2i-1}}=\frac{k+1-i}{n}$, $u_{W_{2i}}=\frac{i}{n}$ for all $1\le i\le k$, $v_{L_1}=\frac{k}{n}$, $v_{L_2}=\frac{k+1}{n}$
  and check that it is a feasible solution. Since $\sum\limits_{i=1}^{n-1} u_{W_i}=\frac{k(k+1)}{n}=\frac{(n-1)(n+1)}{4n}$ the
  proposed lower bound follows.
\end{proof}

So, we are only able to slightly improve the previously known lower bound for $\s_\mathcal{S}(n)$ if the number of
voters is odd. One can easily verify that the given examples have a critical threshold value of
$\frac{\left\lfloor\frac{n^2}{4}\right\rfloor}{n}$.

\begin{conjecture}
  \label{conjecture_spectrum_c_largest}
  For $n\ge 4$ we have $\s_\mathcal{S}(n)=\frac{\left\lfloor\frac{n^2}{4}\right\rfloor}{n}$.
\end{conjecture}

We would like to remark that the simple game defined in the proof of Proposition~\ref{prop_c_largest_lower_bound}
is very far from being the unique one with $\mu(\chi)=\frac{\left\lfloor\frac{n^2}{4}\right\rfloor}{n}$. For the
proof we need that $L_1$, $L_2$ are losing coalitions and that the stated subsets of cardinality two are winning
coalitions. We can construct an exponential number of simple games having a critical $\alpha$ of at least
$\frac{\left\lfloor\frac{n^2}{4}\right\rfloor}{n}$ as follows: Let $L_1'\subsetneq L_1$ and $L_2'\subsetneq L_2$
such that none of the winning coalitions of size two is contained in $L_1'\cup L_2'$ and $\left|L_1'\right|,
\left|L_2'\right|\ge 1$. With this we can specify the coalition $L_1'\cup L_2'$ either as winning or as losing
without violating the other properties. This fact suggests that it might be hard to solve the integer
linear program exactly to determine $\s_\mathcal{S}(n)$ for larger values of $n$, see Section~\ref{sec_ilp_max_alpha}.

Another concept to measure the deviation of a simple game $\chi$ from a weighted voting game is its dimension, i.e.\ the
smallest number $k$ of weighted voting games that $\chi$ is given by their intersection, see e.g.\ \cite{1134.68369}.
It is well known that each simple game has a finite dimension (depending on $n$), see \cite{0765.90030}. Simple games of  
dimension~$1$ coincide with weighted voting games having a critical threshold value of $1$.
The next possible dimension is two, where the critical threshold can be as large as the best known lower bound of 
$\left\lfloor\frac{n^2}{4}\right\rfloor/n$. Thus, there is no direct relation between the dimension of
a simple game and its critical threshold value. To construct such examples we split the voters into sets of cardinality
of at least $\left\lfloor\frac{n}{2}\right\rfloor$, i.e.\ as uniformly distributed as possible,
and assign weight vectors $(1,0)$ to the elements of one such set and $(0,1)$ to the elements from the other set.
Using a quota vector $(1,1)$ we obtain a simple game that satisfies the necessary requirements for a
critical $\alpha$ of at least $\left\lfloor\frac{n^2}{4}\right\rfloor/n$. In other words the dimension of a simple game
is somewhat independent from the critical threshold parameter.

\begin{lemma}
  Let $\chi$ be a simple game with $n$ voters and $\mu(\chi)=\alpha$. If a losing
  coalition of cardinality $k$ exists, then we have $\alpha\le n-k$.
\end{lemma}
\begin{proof}
  Let $S\subsetneq N$ be a losing coalition of cardinality $k$. We use the weights $w_i=0$ for all $i\in S$ and $w_i=1$ for
  all $i\in N\backslash S$. Since $w(N)=n-k$ the weight of each losing coalition is at most $n-k$ and since each winning
  coalition must contain at least one element from $N\backslash S$ their weight is at least $1$.
\end{proof}

\begin{lemma}
  Let $\chi$ be a simple game with $n$ voters and $\mu(\chi)=\alpha$. If the maximum size
  of a losing coalition is denoted by $k$ we have $\alpha\le \max\!\left(1,\frac{k}{2}\right)$.
\end{lemma}
\begin{proof}
  We assign a weight of $1$ to every voter $i$ where $\{i\}$ is a winning coalition and a weight of
  $\frac{1}{2}$ to every other voter. Thus each winning coalition has a weight of at least $1$ and
  each losing coalition a weight of at most $\frac{k}{2}$.
\end{proof}

\begin{corollary}
  \label{cor_largest_alpha_simple_games}
  For each integer $n\ge 3$ we have $\s_\mathcal{S}(n)\le\frac{n}{3}$.
\end{corollary}
\begin{proof}
  Let $\chi$ be a simple game with largest losing coalition of size $k$ and consisting of $n$ voters. If
  $k\le\frac{2n}{3}$ then we have $\mu(\chi)\le \max\left(1,\frac{k}{2}\right)\le \frac{n}{3}$.
  Otherwise, we have $\mu(\chi)\le n-k\le\frac{n}{3}$.
\end{proof}

To further improve Corollary~\ref{cor_largest_alpha_simple_games} some reduction techniques might be useful.

\begin{lemma}
  \label{lemma_reduction_sg}
  If a simple game $\chi$ on $n\ge 2$ voters contains a winning coalition of cardinality one
  then we have $\mu(\chi)\le \s_\mathcal{S}(n-1)$.
\end{lemma}
\begin{proof}
  W.l.o.g.\ let $\{n\}$ be a winning coalition. If $\{1,\dots,n-1\}$ is a losing coalition then $\chi$ is roughly
  weighted using the weights $w_1=\dots=w_{n-1}=0$, $w_n=1$. Otherwise we consider the simple game $\chi'$ arising
  from $\chi$ by dropping voter~$n$. Let $w_1,\dots,w_{n-1}$ be a weighting for $\chi'$ corresponding
  to a threshold value of at most $\s_\mathcal{S}(n-1)$. By choosing $w_n=1$ we can extend this to a valid weighting
  for $\chi$ since every coalition which contains voter $n$ is a winning coalition.
\end{proof}

\bigskip
\bigskip

From now on, we consider complete simple games. To provide a lower bound on $\s_\mathcal{C}(n)$ we consider
a special subclass of complete simple games, i.e., complete simple games with $t=2$ types of voters and a
unique shift-minimal winning vector $(a,b)$ ($r=1$). So, if a coalition contains at least $a$ voters of the
first type and and least $a+b$ members in total, then it is winning, otherwise it is losing.

In the following we will derive conditions on the parameters $a$ and $b$ in order to exclude weighted games,
which would lead to a critical threshold value of $1$. Since the shift-maximal losing vectors depend on a
certain relation between $a$ and $b$, we have to consider two different cases to state the linear program
to determine the critical threshold value.

For $a+b-1\le n_1$ (case 1) the shift-maximal losing vectors are given by
$(a+b-1,0)$, $(a-1,n_2)$ 
and otherwise (case 2) by
$(n_1,a+b-1-n_1)$, $(a-1,n_2)$.

Due to condition~(a).(iii) in Theorem~\ref{thm_csg_characterization} we have $a>0$.  
and $w_1=w_2=\frac{1}{b}$ shows that the game is roughly weighted in this case. 
For $a=n_1$ a quota of $q=n_1n_2+b$ and weights $w_1=n_2$ and $w_2=1$ testify that the
game is weighted. So, we only need to consider $1\le a\le n_1-1$, $0\le b\le n_2-1$.  For $b=0$ the games are weighted via quota
$q=a$ and weights $w_1=1$, $w_2=0$. For $b=1$ the games are weighted via quota $q=an_2+1$ and weights $w_1=n_2$, $w_2=1$. If
$b=n_2$ a quota of $q=a+n_2-1+\frac{a}{n_1+n_2}$ and weights of $w_1=1+\frac{1}{n_1+n_2}$, $w_2=1$ show that these games are
weighted  so that we can assume $2\le b\le n_2-2$ and $n\ge 6$.

To compute $\s_\mathcal{C}(n,r=1,t=2)$ we have to solve the linear program
\begin{eqnarray}
  \min \alpha\quad\quad\text{s.t.}\nonumber\\
  aw_1+bw_2\ge 1\label{lp_1_1}\\
  \alpha-(a+b-1)w_1\ge 0\label{lp_1_2}\\
  \alpha-(a-1)w_1-n_2w_2\ge 0\label{lp_1_3}\\
  w_1\ge w_2\label{lp_1_4}\\
  w_2\ge 0\label{lp_1_5}
\end{eqnarray}
for case~1 and
\begin{eqnarray}
  \min \alpha\quad\quad\text{s.t.}\nonumber\\
  aw_1+bw_2\ge 1\label{lp_2_1}\\
  \alpha-n_1w_1-(a+b-1-n_1)w_2\ge 0\label{lp_2_2}\\
  \alpha-(a-1)w_1-n_2w_2\ge 0\label{lp_2_3}\\
  w_1\ge w_2\label{lp_2_4}\\
  w_2\ge 0\label{lp_2_5}
\end{eqnarray}
for case~2. We would like to remark that we may also include the constraint $\alpha\ge1$. Once it is tight
we have $\alpha=1$, so that we assume $\alpha>1$ in the following.

The optimal solution of these linear programs is attained at a corner of the corresponding polytope which is the solution
of a $3$-by-$3$-equation system arising by combining three of the five inequalities. As notation we use
$A\subset \{\text{\ref{lp_1_1}},\text{\ref{lp_1_2}},\text{\ref{lp_1_3}},\text{\ref{lp_1_4}},\text{\ref{lp_1_5}}\}$ with
$|A|=3$. (Some of these solutions may be infeasible.)  At first, we remark that $w_1=w_2=0$ is infeasible in both cases so
that we assume $|A\cap\{\text{\ref{lp_1_4}},\text{\ref{lp_1_5}}\}|\le 1$.

For case~1 the basic solutions, parameterized by sets of tight inequalities, are given by:
\begin{itemize}
\setlength{\itemsep}{5pt}
 \item[$\{\text{\ref{lp_1_1}},\text{\ref{lp_1_2}},\text{\ref{lp_1_3}}\}$]
       $w_1=\frac{n_2}{an_2+b^2}$, $w_2=\frac{b}{an_2+b^2}$, $\alpha=\frac{n_2(a+b-1)}{an_2+b^2}$, always feasible, e.g.\ we
       have $n_2(b-1)\ge b^2$ due to $b\le n_2-2$ and $b\ge 2$ so that $\alpha\ge 1$ holds.
 \item[$\{\text{\ref{lp_1_1}},\text{\ref{lp_1_2}},\text{\ref{lp_1_4}}\}$] $\alpha=\frac{a+b-1}{a+b}<1$, contradiction
 \item[$\{\text{\ref{lp_1_1}},\text{\ref{lp_1_2}},\text{\ref{lp_1_5}}\}$] $w_1=\frac{1}{a}$, $w_2=0$, $\alpha=\frac{a+b-1}{a}$, always feasible
 \item[$\{\text{\ref{lp_1_1}},\text{\ref{lp_1_3}},\text{\ref{lp_1_4}}\}$] $w_1=\frac{1}{a+b}$, $w_2=\frac{1}{a+b}$, $\alpha=\frac{a-1+n_2}{a+b}$, always feasible
 \item[$\{\text{\ref{lp_1_1}},\text{\ref{lp_1_3}},\text{\ref{lp_1_5}}\}$] $\alpha=\frac{a-1}{a}<1$, contradiction
 \item[$\{\text{\ref{lp_1_2}},\text{\ref{lp_1_3}},\text{\ref{lp_1_4}}\}$] $\alpha=0<1$, contradiction
 \item[$\{\text{\ref{lp_1_2}},\text{\ref{lp_1_3}},\text{\ref{lp_1_5}}\}$] $\alpha=0<1$, contradiction
\end{itemize}
We always have $\frac{a+b-1}{a}>\frac{a+b-1}{a+\frac{b^2}{n_2}}=\frac{n_2(a+b-1)}{an_2+b^2}$ and
\begin{eqnarray*}
 (a+b)\cdot(an_2+b^2)\cdot\left(\frac{a-1+n_2}{a+b}-\frac{n_2(a+b-1)}{an_2+b^2}\right)=b(n_2-b)+a(n_2-b)^2>0.
\end{eqnarray*}
Thus $\alpha=\frac{n_2(a+b-1)}{an_2+b^2}$ is always the minimum value.

\bigskip

For case~2 the basic solutions are given by:
\begin{itemize}
 \item[$\{\text{\ref{lp_2_1}},\text{\ref{lp_2_2}},\text{\ref{lp_2_3}}\}$]
                $w_1=\frac{n_1+n_2+1-a-b}{-a^2-2ab+a+an_1+n_1b+an_2+b}$, $w_2=\frac{n_1+1-a}{-a^2-2ab+a+an_1+n_1b+an_2+b}$,\\
                $\alpha=\frac{n_1n_2-ab+b-a^2+2a+an_1-1-n_1}{-a^2-2ab+a+an_1+n_1b+an_2+b}=:\alpha'$, where we have $w_1\ge w_2$. $\alpha\ge 1$ is equivalent
                to $n_1n_2+a-1-n_1\ge -ab+n_1b+an_2$ which can be simplified to the valid inequality
                $\underset{\ge 1}{\underbrace{(n_1-a)}}\cdot\underset{\ge 1}{\underbrace{(n_2-b-1)}}\ge 1$.
 \item[$\{\text{\ref{lp_2_1}},\text{\ref{lp_2_2}},\text{\ref{lp_2_4}}\}$] $\alpha=\frac{a+b-1}{a+b}<1$, contradiction
 \item[$\{\text{\ref{lp_2_1}},\text{\ref{lp_2_2}},\text{\ref{lp_2_5}}\}$] $w_1=\frac{1}{a}$, $w_2=0$, $\alpha=\frac{n_1}{a}$, always feasible
 \item[$\{\text{\ref{lp_2_1}},\text{\ref{lp_2_3}},\text{\ref{lp_2_4}}\}$] $w_1=\frac{1}{a+b}$, $w_2=\frac{1}{a+b}$, $\alpha=\frac{a-1+n_2}{a+b}$, always feasible
 \item[$\{\text{\ref{lp_2_1}},\text{\ref{lp_2_3}},\text{\ref{lp_2_5}}\}$] $\alpha=\frac{a-1}{a}<1$, contradiction
 \item[$\{\text{\ref{lp_2_2}},\text{\ref{lp_2_3}},\text{\ref{lp_2_4}}\}$] $\alpha=0<1$, contradiction
 \item[$\{\text{\ref{lp_2_2}},\text{\ref{lp_2_3}},\text{\ref{lp_2_5}}\}$] $\alpha=0<1$, contradiction
\end{itemize}
$\alpha'\le \frac{n_1}{a}$ is equivalent to
\begin{eqnarray*}
  \frac{(n_1+1-a)\cdot(a(n_1+1-a)+b(n_1-a))}{a\cdot(a(n_1+n_2+1-a-b)+b(n_1+1-a))} \ge 0
\end{eqnarray*}
and $\alpha'\le\frac{a-1+n_2}{a+b}$ is equivalent to
\begin{eqnarray*}
  \frac{(n_2-b)(a(n_2-b)+b)}{(a+b)\cdot\Big(a(n_2-b)+(a+b)(n_1+1-a)\Big)}\ge 0.
\end{eqnarray*}
Since in both cases all factors are non-negative the respective inequalities are valid and the minimum possible $\alpha$-value is given by $\alpha'$.

\bigskip

To answer the question for the maximum possible $\alpha$ in case~1 depending on $n$ we have to solve the following optimization problem
\begin{eqnarray*}
  \max \frac{a+b-1}{a+\frac{b^2}{n_2}} \quad\quad\text{s.t.}\\
  a+b-1\le n_1\\
  n_1+n_2=n\\
  n_1,n_2\ge 1\\
  1\le a\le n_1-1\\
  2\le b\le n_2-2,
\end{eqnarray*}
where all variables have to be integers. For $z\ge 1$, $x>y>0$ we have $\frac{z-1+x}{z-1+y}>\frac{z+x}{z+y}$. Thus the maximum
is attained at the minimum value of $a$ which is $1$. ($a=1$ also yields the weakest constraint $a+b-1\le n_1$.) Since $1\le a\le n_1-1$
is equivalent to $n_1\ge 2$, which is implied by $a+b-1\le n_1$ via $b\ge 2$, we can drop this constraint. 

If $a+b-1<n_1$ then we could decrease $n_1$ by $1$ and increase $n_2$ by $1$ yielding a larger target
value. Thus we have $a+b-1=n_1$, which is equivalent to $b=n_1$. Using $n_1+n_2=n$ yields $n_2=n-b$. Inserting then yields
the optimization problem
$$
  \max\frac{b}{1+\frac{b^2}{n-b}},\,\,2\le b\le \frac{n-2}{2},
$$
where $b,n\in\mathbb{N}$. Relaxing the integrality constraint results in
$$
  b=\left(\sqrt{n}-1\right)\cdot\frac{n}{n-1}
$$
with optimal value
$$
  \frac{n^{5/2}-2n^2+n^{3/2}}{2n^2-3n^{3/2}+n^{1/2}}\le\frac{\sqrt{n}}{2}
$$
tending to $\frac{\sqrt{n}}{2}$ as $n$ approaches infinity. Since the target function is continuous and there is only one inner
local maximum, the optimal integer solution is either $b=\left\lfloor\left(\sqrt{n}-1\right)\cdot\frac{n}{n-1}\right\rfloor$ or
$b=\left\lceil\left(\sqrt{n}-1\right)\cdot\frac{n}{n-1}\right\rceil$. For $n\ge 9$ also the condition $2\le b\le \frac{n-2}{2}$
is fulfilled. Let us denote the first bound by $\underline{f_1}(n)$ and the second bound by $\overline{f_1}(n)$. In the following
table we compare these bounds with the exact value $\s_\mathcal{C}(n)$, determined using the methods from Section~\ref{sec_ilp_max_alpha},
and $\frac{\sqrt{n}}{2}$.


\begin{center}
  \begin{tabular}{ccccccccccccccc}
  \hline
   $n$ & 9 & 10 & 11 & 12 & 13 & 14 & 15 & 16 \\
   $\underline{f_1}(n)$  &\!1.2727\!&\!1.3333\!&\!1.3846\!&\!1.4286\!&\!1.4667\!&\!1.5000\!&\!1.7143\!&\!1.7727\!\\[1mm]
   $\overline{f_1}(n)$   &\!1.2000\!&\!1.3125\!&\!1.4118\!&\!1.5000\!&\!1.5789\!&\!1.6500\!&\!1.6296\!&\!1.7143\!\\
   $\s_\mathcal{C}(n)$   &\!1.3333\!&\!1.4074\!&\!1.4667\!&\!1.5556\!&\!1.6500\!&\!1.7344\!&\!1.8088\!&\!1.8750\!\\
   $\frac{\sqrt{n}}{2}$&\!1.5000\!&\!1.5811\!&\!1.6583\!&\!1.7320\!&\!1.8028\!&\!1.8708\!&\!1.9365\!&\!2.0000\!\\ 
   \hline
  \end{tabular}
\end{center}

\bigskip

In case~2 we obtain the optimization problem
\begin{eqnarray*}
  \max \frac{n_1n_2-ab+b-a^2+2a+an_1-1-n_1}{-a^2-2ab+a+an_1+n_1b+an_2+b} \quad\quad\text{s.t.}\\
  a+b-1\ge n_1+1\\
  n_1+n_2=n\\
  n_1,n_2\ge 1\\
  1\le a\le n_1-1\\
  2\le b\le n_2-2,
\end{eqnarray*}
For $a>1$ we can check that decreasing $a$, $n_1$ and increasing $b$, $n_2$ by $1$ does not decrease the target
value. Thus we can assume $a=1$ in the optimal solution so that the target function simplifies to
$\frac{n_1n_2}{n_1(b+1)+(n_2-b)}=\frac{n_2}{b+1+\frac{n_2-b}{n_1}}$. Decreasing $b$ by $1$ increases this target
function so that either $a+b-1\ge n_1+1$ or $b\ge 2$ is tight. In the latter case we would have $n_1\le 1$, which contradicts
$1=a\le n_1-1$. Thus, we have $a+b-1= n_1+1$ in the optimum which is equivalent to $b=n_1+1$. Inserting this and $n_2=n-n_1$
yields the target function
$$
  \frac{n-b+1}{b+1+\frac{n-2b+1}{b-1}}
$$
having the non-negative optimal solution of $b=\frac{1+\sqrt{1+n^3-2n}}{n}$ with target value
$$
  \frac{1}{2}\cdot\frac{\sqrt{n^3+1-2n}-(n-1)}{n-1}\le\frac{\sqrt{n}}{2}
$$
tending to $\frac{\sqrt{n}}{2}$ as $n$ approaches infinity. If the other inequalities are fulfilled, then rounding up or
down yields the optimal integral solution (in this case; not in general). In both cases the conditions $2\le b\le n_2-2$,
$1=a\le n_1-1$ are fulfilled for $n\ge 9$. We produce a similar table as before:

\begin{center}
  \begin{tabular}{ccccccccccccccc}
  \hline
   $n$ & 9 & 10 & 11 & 12 & 13 & 14 & 15 & 16 \\
   $\underline{f_2}(n)$ &\!1.1667 &\!1.2308\!&\!1.2857\!&\!1.3333\!&\!1.3750\!&\!1.4118\!&\!1.4444\!&\!1.6250\!\\[1mm]
   $\overline{f_2}(n)$  &\!1.0588 &\!1.1667\!&\!1.2632\!&\!1.3500\!&\!1.4286\!&\!1.5000\!&\!1.5652\!&\!1.5484\!\\
   $\s_\mathcal{C}(n)$     &\!1.3333 &\!1.4074\!&\!1.4667\!&\!1.5556\!&\!1.6500\!&\!1.7344\!&\!1.8088\!&\!1.8750\!\\
   \hline
  \end{tabular}
\end{center}

\medskip

\begin{conjecture}
\label{conjecture_spectrum_c_largest_complete_simple_game}
 $$
   \s_\mathcal{C}(n)\in\Theta\!\left(\sqrt{n}\right).
 $$
\end{conjecture}

\bigskip

So far we do not know any examples of complete simple games with a critical threshold value larger than
$\max\!\left(1,\frac{\sqrt{n}}{2}\right)$. We will prove Conjecture~\ref{conjecture_spectrum_c_largest_complete_simple_game}
for some special classes of complete simple games. An important class, used by many real-world voting systems, is given by
the so-called games with consensus, i.e.\ intersections of a weighted voting game and a symmetric game $[q';1,\dots,1]$,
see e.g.\ \cite{1087.91005,0856.90029}. The voting procedure for the council of the European Union based on the Treaty of
Nice consists of such a consensus, i.e.\ at least $14$ (or $18$, if the proposal was not made by the commission) of the
countries must agree. (The two other ingredients are a majority of the voting weights and a majority of the population.)
Concerning the distribution of power in the European Union we refer the interested reader to e.g.\ \cite{1110.90067}.

\begin{lemma}
  \label{lemma_games_with_consensus}
  The critical threshold value $\mu(\chi)$ of a complete simple game $\chi\in\mathcal{C}_n$ with consensus, given as
  the intersection of $[q;w_1,\dots,w_n]$ and $[q';1,\dots,1]$, is at most $\sqrt{n}$.
\end{lemma}
\begin{proof}
  If $q'\ge \sqrt{n}$ we take weights of $\frac{1}{\sqrt{n}}$ for all voters so that each winning coalition has a weight of at
  least one and the grand coalition a weight of $\sqrt{n}$. In the other cases we take weights $\frac{w_i}{q}$ for the voters
  so that each winning coalition has a weight of at least $1$. W.l.o.g.\ we assume $w_i\le q$ so that the new weights
  are at most $1$. A losing coalition with weight larger than one must fail the criterion of the symmetric game so that
  it consists of less than $\sqrt{n}$ members. Thus the weight of each losing coalition is less than $\sqrt{n}$.
\end{proof}

For large consensus $q'$ the critical threshold value is bounded from above by $\frac{n}{q'}$, since we can assign weights
of $\frac{1}{q'}$ to all voters. We remark that complete simple games 
$\left((n_1,n_2),(m_1,m_2)\right)$ with two equivalence classes of voters and one
shift-minimal winning vector are games with consensus and thus have a dimension of at most two\footnote{Complete simple games with
one shift-minimal winning vector and more than two equivalence classes of voters can have dimensions larger than two and as large
as $\frac{n}{4}$ \cite{1151.91021}.}. As representation we may use the intersection of $[m_1+m_2;1,\dots,1]$ and 
$[m_1n_2+m_2;n_2,\dots,n_2,1,\dots,1]$.

\begin{lemma}
  The critical threshold value $\mu(\chi)$ of a complete simple game $\chi\in\mathcal{C}_n$ with two types of voters is
  at most $\sqrt{n}+1$.
\end{lemma}
\begin{proof}
  If $\chi$ has only one shift-minimal winning vector we can apply Lemma~\ref{lemma_games_with_consensus}. Since complete simple
  games with less than four voters are weighted we can assume $n\ge 4$. So let $m_1=(a,b)$ the shift-minimal winning
  vector with maximal $a$ and $m_2=(c,d)$ the shift-minimal winning vector with minimal $c$. Depending on
  the values of $a$ and $c$ we will provide suitable weights $w_1$ and $w_2$ such that each winning coalition has a weight of at least $q>0$ and 
  each losing coalition has a weight of at most $q\cdot(\sqrt{n}+1)$, i.e.\ the proposed weights have to be normalized in order to fit
  into the framework of a quota $q=1$.
  
  If $c\ge 1$ we set $w_1=\sqrt{n}$ and $w_2=1$. Every shift-minimal winning vector $(e,f)\neq(a,b)$ must
  fulfill $c\le e\le a$  due to the definition of $a,c$ and $e+f\ge a+b+1$ since otherwise $(a,b)$
  would not be a shift-minimal winning vector. With this we have
  $$
    ew_1+fw_2\ge ew_1+(a+b+1-e)w_2\ge c\sqrt{n} +(a+b+1-c).
  $$
  Similarly, we obtain
  $$
    aw_1+bw_2=c\sqrt{n}+a-c+b+(\underset{\ge 1}{\underbrace{a-c}})\cdot(\underset{\ge 1}{\underbrace{\sqrt{n}-1}})\ge c\sqrt{n} +(a+b+1-c).
  $$
  Thus it suffices to show that each losing coalition has a weight of at most
  $$\left(c\sqrt{n} +(a+b+1-c)\right)\cdot\left(\sqrt{n}+1\right)\ge n+a\sqrt{n}+b\sqrt{n}.$$
  Let $(g,h)$ be a losing coalition so that $(g,h)\nsucceq(a,b)$ and $(g,h)\nsucceq(c,d)$.
  If $g\le c$ then $h\le n_2\le n-a$ and we have
  $$
    gw_1+hw_2\le c\sqrt{n}+n-a\le n+a\sqrt{n}.
  $$
  If $g\ge a$ then $g+h\le a+b-1$ since otherwise $(g,h)\succeq(a,b)$.
  With this we have
  $$
    gw_1+hw_2\le (a+b-1)\sqrt{n}\le a\sqrt{n}+b\sqrt{n}.
  $$
  If $c\le g<a$ then $g+h\le c+d-1$ since otherwise $(g,h)\succeq(c,d)$.
  With this we have
  $$
    gw_1+hw_2\le (a-1)\sqrt{n}+(c+d-a)\le n+a\sqrt{n}.
  $$
  
  \medskip
  
  If $c=0$ we set $w_1=\sqrt{d}$, where $d\ge a+b+1\ge 2$, and $w_2=1$. Let $(e,f)$ be a winning
  and  $(g,h)$ be a losing coalition. Similarly, as before we have $e+f\ge a+b$ so that
  $$
    ew_1+fw_2\ge\sqrt{d}+a+b-1.
  $$
  It suffices to show that each losing coalition has a weight of at most
  \begin{eqnarray*}
    \left(\sqrt{d}+a+b-1\right)\cdot\left(\sqrt{n}+1\right) &\ge& \underset{\ge d}{\underbrace{\sqrt{dn}-\sqrt{n}+\sqrt{d}}}+(a+b)\sqrt{n}
    +\underset{\ge 0}{\underbrace{a+b-1}}\\
    &\ge& d+(a+b)\sqrt{n}.
  \end{eqnarray*}
  If $g\ge a$ then $g+h\le a+b-1$, since otherwise $(g,h)\succeq(a,b)$, and we have
  $$
    gw_1+hw_2\le (a+b-1)\sqrt{d}\le (a+b)\sqrt{n}.
  $$
  If $c\le g<a$ then $g+h\le c+d-1$, since otherwise $(g,h)\succeq(c,d)$, and we have
  $$
    gw_1+hw_2\le (a-1)\sqrt{d}+(c+d-a)\le a\sqrt{n}+d.
  $$
\end{proof}

We remark that complete simple games with one type of voters are weighted and thus have a critical threshold value of $1$.

\begin{lemma}
  The critical threshold value $\mu(\chi)$ of a complete simple game $\chi\in\mathcal{C}_n$ with one shift-minimal winning
  vector $\widetilde{a}$ is at most $\sqrt{n}$.
\end{lemma}
\begin{proof}
  By $(n_1,\dots,n_t)$ we denote the numbers of voters in the $t\ge 2$ equivalence classes
  of voters and by $(a_1,\dots,a_t)$ the unique shift-minimal winning vector $\widetilde{a}$.
  
  If $\sum\limits_{i=1}^t a_i\ge \sqrt{n}$ we set $w_i=\frac{1}{\sqrt{n}}$ for all $1\le i\le t$ and have $w(\widetilde{a})\ge 1$.
  Since with these weights we have $w(N)\le \sqrt{n}$, every losing coalition has a weight of at most $\sqrt{n}$ and
  we have a critical threshold value of at most $\sqrt{n}$.
  
  In the remaining cases we have $\sum\limits_{i=1}^t a_i\le \sqrt{n}$. Due to condition (a)(iii) of
  Theorem~\ref{thm_csg_characterization} we have $a_1\ge 1$.  We set $w_1=1$ and $w_2=\dots=w_t=0$ and
  have $w(\widetilde{a})\ge 1$. For every losing vector $\widetilde{l}=(l_1,\dots,l_t)$ we have
  $l_1<\sum\limits_{i=1}^t a_i\le\sqrt{n}$ since otherwise we would have $\widetilde{a}\prec \widetilde{l}$. Thus
  each losing coalition has a weight of at most $\sqrt{n}$ and the critical threshold value is bounded from above
  by $\sqrt{n}$ in this case.
  %
\end{proof}

So, we have an upper bound of $\sqrt{n}$ for the critical threshold value for complete simple games on $n$~voters in several
subcases. For the general case of Conjecture~\ref{conjecture_spectrum_c_largest_complete_simple_game} we can provide only a
first preliminary bound showing that $\s_\mathcal{C}(n)$ asymptotically grows slower than $\s_\mathcal{B}(n)$ so that the
maximum critical threshold value in some sense states that complete simple games are \textit{nearer} to (roughly)
weighted voting games than simple games.

\begin{theorem}
  The critical threshold value $\mu(\chi)$ of a complete simple game $\chi\in\mathcal{C}_n$ is in
  $O\!\left(\frac{n\cdot\log\log n}{\log n}\right)$.
\end{theorem}
\begin{proof}
  As weights we choose a slowly decreasing geometric series $w_i=q^{i-1}$ for all $1\le i\le n$ where
  $q=1-\frac{\log n}{n\cdot\log \log n}$. With this we have $0\le q< 1$ and $\frac {1}{1-q}=\frac{n\cdot\log\log n}{\log n}$.
  Now, let $W$ be a winning coalition with the minimum weight and $L$ be a losing coalition with the
  maximum weight. In the following we will show $\frac{w(L)}{w(W)}\le\frac{n\cdot\log\log n}{\log n}$. To deduce this
  bound we will compare the weights of a few subsets of consecutive voters. In order to keep the necessary number of such
  subsets small, we set $\widetilde{W}:=W\backslash (W\cap L)$ and $\widetilde{L}:=L\backslash (W\cap L)$, i.e.\ we technically
  remove common voters. We remark that $\tilde{W}$ needs not be a winning coalition.
  Due to the inequality
  $$
    \frac{x}{y}\ge\frac{x+c}{y+c}
  $$
  for $x\ge y>0$ and $c\ge 0$ it suffices to provide an upper bound for $\frac{w(\widetilde{L})}{w(\widetilde{W})}$.
 
  At first, we consider the case when $W$ is lexicographically larger than $L$. Let $j$ be the voter with the minimal index
  (and so the maximal weight) in
  $\widetilde{W}$. With this we set $W'=\{j\}$, $L'=\{j+1,\dots,n\}$ and have $w(\widetilde{W})\ge w(W')$,
  $w(\widetilde{L})\le w(L')$ so that $\frac{w(L)}{w(W)}$ is upper bounded by
  $$
    \frac{w(\widetilde{L})}{w(\widetilde{W})}\le\frac{w(L')}{w(W')}=\frac{q(1-q^{n-j})}{1-q}\le \frac{1}{1-q}=\frac{n\cdot\log\log n}{\log n}.
  $$
  
  \noindent
  If $W$ is lexicographically smaller than $L$ then let $j$ be an index with $|\underset{=:k_1}{\underbrace{\widetilde{W}}\cap\{1,\dots,j\}}|$ $>
  |\underset{=:k_2}{\underbrace{\widetilde{L}}\cap\{1,\dots,j\}}|$. With this we set $L':=\{1,\dots,k_2\}\cup\{j+1,\dots,n\}$ and
  $W':=\{j-k_1+1,\dots,j\}$
  fulfilling $w(\widetilde{W})\ge w(W')$ and $w(\widetilde{L})\le w(L')$. Since $k_1\ge k_2\ge 1$,
  $$
    w(L')=\sum_{i=1}^{k_2}q^{i-1}+\sum_{i=j+1}^n q^{i-1}=\frac{1-q^{k_2}}{1-q}+q^j\cdot\frac{1-q^{n-j}}{1-q}
  $$
  and $w(W')=q^{j-k_1+1}\cdot\frac{1-q^{k_1}}{1-q}\ge q^{j-k_1+1}$ we have
  $$
    \frac{w(L)}{w(W)}\le\frac{w(\widetilde{L})}{w(\widetilde{W})}\le
    \frac{w(L')}{w(W')}\le q^{k_1-j-1}+\frac{q^j\cdot\frac{1}{1-q}}{q^{j-k_1+1}}\le q^{-j}+\frac{1}{1-q}\le q^{-n}+\frac{1}{1-q}.
  $$
  To finish the proof we show $q^{-n}\in O\!\left(\frac{n\cdot\log \log n}{\log n}\right)$. From $\frac{x}{1+x}\le \log (1+x)\le x$
  for $x>-1$ we conclude $2x\ge\frac{x}{1-x}\ge -\log(1-x)\ge x$ for $\frac{1}{2}\le x\le 1$. Thus for large enough $n$ we have
  $$
  \log\!\left(q^{-n}\right)\le n\cdot \left(-\log\left(1-\frac{\log n}{n\cdot\log\log n}\right)\right)\le n\cdot \frac{2\log n}{n\cdot\log\log n}\le
  \frac{2\log n}{\log\log n}
  $$
  and $\frac{2\log n}{\log\log n}\le \log n-\log\log n+\log\log\log n=\log\!\left(\frac{n\cdot\log \log n}{\log n}\right)$.
\end{proof}

\bigskip

In the context of the conjectured upper bound of $O(\sqrt{n})$ for $c_\mathcal{C}(n)$ we find it remarkable that the 
cost of stability $CoS(f)$ of any super-additive, see Definition~\ref{def_super_additive}, Boolean game $f\in\mathcal{B}_n$
is upper bounded by $\sqrt{n}-1$, see \cite{Bachrach:2009:CSC:1692490.1692502}. If $f$ is additionally anonymous, then the
authors have proven the tighter bound $CoS(f)\le 2$. This coincides with the situation for the critical threshold value. Here
we may consider the super-additive anonymous Boolean game $f\in\mathcal{B}_n$, where coalitions of size
$\left\lceil\frac{n+1}{2}\right\rceil$ are winning and the grand coalition $N$ is losing.

\section{An integer linear programming approach to determine the maximal critical threshold value}
\label{sec_ilp_max_alpha}

\noindent
In principle it is possible to determine the maximal critical threshold value $\s_\mathcal{S}(n)$ for a given integer $n$ by
simply solving the stated linear program from Proposition~\ref{prop_lp_critical_alpha_simple_game} for all simple games $\chi\in\mathcal{S}_n$.
Since for $n\le 8$ there are $1$, $4$, $18$, $166$, $7\,579$, $7\,828\,352$, $2\,414\,682\,040\,996$, and $56\,130\,437\,228\,687\,557\,907\,786$
simple games,
an exhaustive search seems to be hopeless even for moderate $n$ (of course theoretical results may help to reduce the number of simple games
which need to be checked). For $n=9$ only the lower bound $10^{42}$ is known.

So, alternatively we will formulate $\s_\mathcal{S}(n)$ as the solution of an optimization problem in the following to avoid exhaustive
enumeration. It is possible to describe the set of monotone Boolean functions as integer points of a polyhedron, see e.g.\
\cite{inverse_power_index_problem}: For each subset $S\subseteq N$ we introduce a binary variable $x_S$ and use the constraints
$x_\emptyset=0$, $x_N=1$, and $x_{S\backslash\{i\}}\le x_S$ for all $\emptyset\neq S\subseteq N$, $i\in S$ to model a simple game via
$\chi(S)=x_S$. (We have to remark that this ILP formulation is very \textit{symmetric}.) In this framework it is very easy to add
additional restrictions. Methods to restrict the underlying games to complete simple games or weighted voting games are outlined
in \cite{inverse_power_index_problem}. The restriction to e.g.\ \textit{proper} simple games can be modeled via
$x_S+x_{N\backslash S}\le 1$ for all $S\subseteq N$. Similarly, \textit{strong} simple games can be modeled by using the
constraints $x_S+x_{N\backslash S}\ge 1$ for all $S\subseteq N$.

So the problem of determining $\s_\mathcal{S}(n)$ can be stated as the following optimization problem: Maximize over all simple games
with $n$ voters the minimum $\alpha$ of the linear program~(\ref{lp_critical_alpha_value}). Since this is a two-level
optimization problem, we have to reformulate the problem in order to apply integer linear programing techniques.

In order to determine $\s_\mathcal{S}(n)$ we cannot maximize $\alpha$ directly since we have $\chi\in\mathcal{T}_{\lambda\alpha}\cap \mathcal{S}_n$
for all $\lambda\ge 1$ if $\chi\in\mathcal{T}_{\alpha}\cap\mathcal{S}_n$. To specify the minimum value $\alpha$ for a given simple game $\chi$
we can also maximize its corresponding dual linear program of~(\ref{lp_critical_alpha_value}) whose optimal solution is $\alpha$.

If we drop the restriction $\alpha\ge 1$ and assume $w_i\ge 0$, the dual program for a simple game $\chi$ is given by 
$$
\begin{array}{rl}
  \text{Max} & \sum\limits_{S\in W} u_S  \\
   & \sum\limits_{S\in W:i\in S} u_S-\sum\limits_{S\in L:i\in S} v_S \le 0\quad \forall i\in N\\
   & \sum\limits_{S\in L} v_S \le 1\\
   & u_S\ge 0 \quad\forall S\in W\\
   & v_S\ge 0 \quad\forall S\in L,
\end{array}
$$
where $W$ denotes the set of winning coalitions and $L$ denotes the set of losing coalitions. As outlined in Section~\ref{sec_preliminaries} the optimal target value
$\sum\limits_{S\in W} u_S$ might take values smaller than $1$ (but being non-negative) which correspond to a critical threshold value of $\mu(\chi)=1$.

The next step is to replace the externally given sets $W$ and $L$ by variables such that the possible sets correspond to simple games. Using our previously defined
binary variables $x_S$ this is rather easy:
$$
\begin{array}{rl}
  \text{Max} & \sum\limits_{S\subseteq N} x_S\cdot u_S  \\
   & \sum\limits_{\{i\}\subseteq S\subseteq N} x_S\cdot u_S-\sum\limits_{\{i\}\subseteq S\subseteq N} \left(1-x_S\right)\cdot v_S \le 0\quad \forall i\in N\\
   & \sum\limits_{S\subseteq N} \left(1-x_S\right)\cdot v_S \le 1\\
   & x_\emptyset=0\\
   & x_N=1\\
   & x_{S\backslash\{i\}}\le x_S\quad  \forall \emptyset\neq S\subseteq N\\
   & u_S\ge 0 \,\,\quad\quad\quad\forall S\subseteq N\\
   & v_S\ge 0 \,\,\quad\quad\quad\forall S\subseteq N\\
   & x_S\in\{0,1\}\,\,\quad\forall S\subseteq N,
\end{array}.
$$

The problem is a quadratically constrained quadratic program (QCQP) with binary variables or more generally a mixed-integer quadratically
constrained program (MIQCP). There are solvers, like e.g.\ \texttt{ILOG CPLEX}, that can deal with these
problems efficiently whenever the target function and the constraints are convex. Unfortunately, neither our target function nor the
feasibility set is convex.
Thus in order to solve this optimization problem directly, we have to utilize a solver that can deal with non-convex mixed-integer
quadratically constrained programs like e.g.\ \texttt{SCIP}, see e.g.\ \cite{BertholdGleixnerHeinzVigerske2011TR,BeHeVi09}\footnote{We
have to remark that currently SCIP is not capable of solving the stated problem without further information because there are some problems
if the intermediate LP relaxations are unbounded. So one has to provide upper and lower bounds for the continuous variables $u_S$ and $v_S$.}.

This works in principle, but problems become computationally infeasible very quickly. By disabling preprocessing we can force \texttt{SCIP}
to use general MIQCP-tech\-niques. Solving the problem Boolean functions with $f(\emptyset)=0$ and $n=3$ took 0.07~seconds and 43~b\&b-nodes, for $n=4$
it took 8.45~seconds and 15770~b\&b-nodes, and for $n=5$ we have aborted the solution process after 265~minutes and $1.6\cdot 10^6$~nodes, where more
than 33~GB of memory was used. 

By enabling preprocessing \texttt{SCIP} is able to automatically find a reformulation as a binary linear program. This way \texttt{SCIP} can solve the
instance for $n=8$ in 2.9~seconds in the root node but will take more than 211~minutes, 373000~nodes, and 1.8 GB of memory to solve the instance for $n=9$. 


Since often binary linear programs are easier to solve than binary quadratic problems, we want to reformulate our binary quadratic optimization problem into
a binary linear one. There are several papers dealing with reformulations of MIQCPs into easier problems, see e.g.\ \cite{reformulation}. Here we want to present
a custom-tailored approach based on some techniques that are quite standard in the mixed integer linear programing community (but we will outline them
nevertheless). Using this formulation, \texttt{SCIP} needed only 18.72~seconds to solve the instance for $n=15$ without applying branch\&bound. We would like to
remark that \texttt{CPLEX} was even faster using only 5.61~seconds of computation time.

A quite general technique to get rid of logical implications are so called Big-M constraints, see e.g.\ \cite{Koch2004b}. To
explain the underlying concept we consider a binary variable $x\in\{0,1\}$, a real-valued variable $y$, and a \emph{conditional} inequality
$y\le c$ for a constant $c$, which only needs to be satisfied if $x=1$. The idea is to use this inequality, but to modify its right-hand side with a constant
times $(1-x)$:
$$
  y\le c +(1-x)\cdot M.
$$
For $x=1$ this inequality is equivalent to the desired \emph{conditional} inequality. Otherwise the new inequality is equivalent to $y\le c+M$, which is
satisfied if $M$ is large enough. Given a known upper bound $y\le u$, where possibly $u\gg c$, it suffices to choose $M=u-c$.

Now we want to apply this technique in a more sophisticated way, to remove the non-linear term $x_S\cdot u_S$, where $x_S\in\{0,1\}$ and
$u_S\in [0,\beta]$. We replace the term $x_S\cdot u_S$ by the variable
$z\ge 0$ using the constraints $z\le \beta x_S$, $z\le u_S$, and $z\ge u_S-\beta \left(1-x_S\right)$. If $x_S=1$ these inequalities state
that $z=x_S\cdot u_S=u_S$ must hold and for $x_S=0$ they imply $z=x_S\cdot u_S=0$. Thus one extra variable and three additional inequalities
are necessary for each term of the form $x_S\cdot u_S$ or $x_S\cdot v_S$. The LP relaxation gets worser with increasing $\beta$, the so-called
Big-M constant. Of course in general, it may be hard to come up with a concrete bound $\beta$. In our case it is not too hard to prove
$u_S,v_S\le 1$: If $x_T=0$ then from $v_S\ge 0$ for all $S\subseteq N$ and $\sum\limits_{S\subseteq N} \left(1-x_S\right)\cdot v_S \le 1$
we conclude $v_T\le 1$. Otherwise $v_T$ does not occur anywhere in the optimization problem and $v_T\le 1$ is a valid inequality. Similarly,
if $x_T=0$ then $u_T$ does not appear anywhere and on the other hand for $x_T=1$ we have
$$
  u_T\le \sum\limits_{\{i\}\subseteq S\subseteq N} \left(1-x_S\right)\cdot v_S\le \sum\limits_{S\subseteq N} \left(1-x_S\right)\cdot v_S \le 1.
$$

Due to the special structure of our problem we can reformulate our problem without additional variables and fewer additional constraints.
The main idea is to use the term $u_S$ instead of $x_S\cdot u_S$ and to ensure that we have $u_S=0$ for $x_S=0$. Similarly, we replace the
products $\left(1-x_S\right)\cdot v_S$ by $v_S$ and ensure that we have $v_S=0$ if $x_S=1$.

\begin{eqnarray}
  \max && \sum\limits_{S\subseteq N} u_S\\
  x_\emptyset &=& 0\\
  x_N &=& 1\\
  x_S-x_{S\backslash\{i\}} &\ge& 0\quad\forall \emptyset\neq S\subseteq N, i\in S\label{ie_sg}\\
  \sum\limits_{\{i\}\subseteq S\subseteq N} u_S-\sum\limits_{\{i\}\subseteq S\subseteq N} v_S &\le& 0\quad\forall i\in N\label{ie_dual_1}\\
  \sum\limits_{S\subseteq N} v_S &\le& 1\label{ie_dual_2}\\
  u_S&\le& x_S\quad\forall S\subseteq N \label{ie_implication_1}
\end{eqnarray}
\begin{eqnarray}
  v_S&\le& 1-x_S \forall S\subseteq N\label{ie_implication_2}\\
  x_S &\in & \{0,1\}\quad\forall S\subseteq N\\
  u_S &\ge & 0\quad\forall S\subseteq N\\
  v_S &\ge & 0\quad\forall S\subseteq N
\end{eqnarray}
Inequalities~(\ref{ie_dual_1}) and (\ref{ie_dual_2}) capture the dual linear program to bound $\alpha=\sum\limits_{S\subseteq N} u_S$ 
from above. Inequality~(\ref{ie_implication_1}) models the implication that $u_T$ is zero if $x_T=0$. In the other case where $x_T=1$
the inequality $u_T\le 1$ is redundant since we have for an $i\in T$ ($x_\emptyset=0$) the inequality
$\sum\limits_{\{i\}\subseteq S\subseteq N}u_S-\sum\limits_{\{i\}\subseteq S\subseteq N}\le 0$ from which we conclude 
$x_T\le 1$ using $x_S\ge 0$ and $\sum\limits_{S\subseteq N} v_S\le 1$. Inequalities of that type are called Big-M inequalities, where
we have an \textit{Big-M} of $1$ in our two cases. (See Inequality~(\ref{ie_implication_4}) for an example with a \textit{Big-M}
constant larger than $1$.) Similarly, Inequality~(\ref{ie_implication_2}) models the implication that $v_T$ is zero if $x_T=1$.
In the other case where $x_T=0$ we have the redundant inequality $v_T\le 1$.

The optimum target value of this ILP is the desired value $\s_\mathcal{S}(n)$ for each integer $n$. We have to remark that our
modeling of the set of simple games is highly symmetric and each solution comes with at least $n!$ isomorphic solutions which
is an undesirable feature for an ILP model. With the stated ILP model we were able to computationally prove
Conjecture~\ref{conjecture_spectrum_c_largest} for $n\le 9$ taking less than 37~seconds for $n=7$, less than $279$~seconds for
$n=8$ but already $66224$~seconds and  $161898779$~branch\&bound nodes for $n=9$. For $n=10$ we have computationally obtained
the bounds $\frac{5}{2}\le \s_\mathcal{S}(10)\le 3$ from an aborted ILP solution process. (The LP relaxation gives only the
relatively poor upper bound of $\frac{n-1}{2}$.)

We would like to remark that we can enhance this ILP formulation a bit. Since we have $\s_\mathcal{S}(n+1)\ge \s_\mathcal{S}(n)$
we may apply Lemma~\ref{lemma_reduction_sg} and require $x_{\{i\}}=0$ for all $1\le i\le n$, where $n\ge 2$.

If we replace conditions~(\ref{ie_sg}) by those for complete simple games we can determine the exact values $\s_\mathcal{C}(n)$
for $n\le 16$: $\s_\mathcal{C}(1)=\s_\mathcal{C}(2)=\s_\mathcal{C}(3)=\s_\mathcal{C}(4)=\s_\mathcal{C}(5)=\s_\mathcal{C}(6)=1$,
$\s_\mathcal{C}(7)=\frac{8}{7}$, $\s_\mathcal{C}(8)=\frac{26}{21}$, $\s_\mathcal{C}(9)=\frac{4}{3}$, $\s_\mathcal{C}(10)=\frac{38}{27}$,
$\s_\mathcal{C}(11)=\frac{22}{15}$, $\s_\mathcal{C}(12)=\frac{14}{9}$,  $\s_\mathcal{C}(13)=\frac{33}{20}$,
$\s_\mathcal{C}(14)=\frac{111}{64}$, $\s_\mathcal{C}(15)=\frac{123}{68}$, and $\s_\mathcal{C}(16)=\frac{15}{8}$. 

We would like to remark that the LP relaxation gives only the poor upper bound $\s_\mathcal{C}(n)\le \frac{n-1}{2}$.

\section{The spectrum of critical threshold values}
\label{sec_spectrum}

\noindent
In sections \ref{sec_maximal_critical_alpha} and \ref{sec_ilp_max_alpha} we have considered the maximum critical threshold value for
several classes of games. By $Spec_\mathcal{S}(n)$ we denote the entire set of possible critical threshold values of simple games on $n$
voters. Similarly, we define $Spec_\mathcal{B}(n)$ as the set of possible critical threshold values for Boolean functions
$f:2^N\rightarrow\{0,1\}$ with $f(\emptyset)=0$ and $Spec_\mathcal{C}(n)$ as the set of possible critical threshold values for complete
simple games on $n$ voters. In this section we will provide a superset for the spectrum using known information of the set of
possible determinants of 0/1 matrices. In order to compute the exact sets for small values of $n$ we modify the presented integer linear 
programming approach for the determination of the maximum critical threshold value to that end.

By considering null voters we conclude $Spec_\mathcal{S}(n)\subseteq Spec_\mathcal{S}(n+1)$,
$Spec_\mathcal{B}(n)\subseteq Spec_\mathcal{B}(n+1)$, and $Spec_\mathcal{C}(n)\subseteq Spec_\mathcal{C}(n+1)$. Due to the inclusion
of the classes of games we obviously have $Spec_\mathcal{C}(n)\subseteq Spec_\mathcal{S}(n)\subseteq Spec_\mathcal{B}(n)$ for all $n\in\mathbb{N}$.

Principally, it is possible to determine the sets $Spec_\mathcal{S}(n)$ for small numbers of voters by exhaustive enumeration of
all simple games. As mentioned in the previous section this approach is very limited due to the quickly increasing number of simple
games. In  \cite{hierarchies} the authors have determined $Spec_\mathcal{S}(n)$ for all $n\le 6$ by some theoretical reductions and 
exhaustive enumeration on the restricted set of possible games.

In this section we want to develop an approach based on integer linear programming to determine the spectrum and to utilize results on
Hadamard's maximum determinant problem to obtain a superset of the spectrum. For the latter let us consider the linear
program~(\ref{lp_critical_alpha_value}) determining the critical threshold value of a Boolean function with $f(\emptyset)=0$. 
Each element of the spectrum $Spec_\mathcal{B}(n)$ appears as the optimal solution of this linear program for a certain
Boolean function $f$. If inequality $\alpha\ge 1$ is attained with equality in the optimal solution, the critical threshold
value is $1$. So we may drop this inequality and consider only those functions $f$ where the linear
program~(\ref{lp_critical_alpha_value}) without the inequality $\alpha\ge 1$ has an optimal solution, which is then attained in
a corner. Thus there are subsets $W_1,\dots,W_k\subseteq N$, where $0\le k\le n+1$, with $\sum\limits_{j\in W_i} w_j= 1$ and
$n+1-k$ subsets $L_1,\dots,L_{n+1-k}\subseteq N$ with $-\alpha+\sum\limits_{j\in L_i} w_j= 0$ such that the entire linear
equation system has a unique solution. (We remark that $k=0$ and $k=n+1$ lead to infeasible solutions.)

Writing this equation system in matrix notation $A\cdot(w_1,\dots,w_n,\alpha)^T=b$ we can use Cramer's rule to state
$$
  \alpha=\frac{\det\!\left(A_{\alpha}\right)}{\det(A)},
$$
where $A_{\alpha}$ arises from $A$ by replacing the rightmost column by $b$. Since $A_{\alpha}$ is a $0/1$-matrix we can use an improved
version of Hadamard's bound and have
$$
  \left|\det\!\left(A_{\alpha}\right)\right|\le \frac{(n+2)^{(n+2)/2}}{2^{n+1}},
$$
see e.g.\ \cite{0249.15003}. If we multiply the rightmost column of $A$ by $-1$, which changes the determinant by a factor of
$(-1)^{n+1}$ then it becomes a $0/1$-matrix too and we conclude
$$
  \left|\det(A)\right|\le \frac{(n+2)^{(n+2)/2}}{2^{n+1}}.
$$

\begin{lemma}
  For each $\alpha\in Spec_\mathcal{B}(n)$ there are coprime integers $1\le q< p \le \left\lfloor\frac{(n+2)^{(n+2)/2}}{2^{n+1}}\right\rfloor$
  with $\alpha=\frac{p}{q}$. 
\end{lemma}

For specific $n$ the uppers bounds on the determinant of $0/1$-matrices can be improved. The exact values for the maximum determinant of a
$n\times n$ binary matrix for $n\le 17$ are given by $1, 1, 2, 3, 5, 9, 32, 56, 144, 320, 1458, 3645$,
$9477, 25515, 131072,$ $327680, 1114112$,
see e.g.\ sequence A003432 in the on-line encyclopedia of integer sequences and the references therein.

Another restriction on the possible critical threshold values is obviously given by the maximum values, i.e.\ $\mu(\chi)\le \s_\mathcal{B}(n)$
(or $\mu(\chi)\le \s_\mathcal{S}(n)$ for simple games, $\mu(\chi)\le \s_\mathcal{C}(n)$ for complete simple games. Further restrictions come
from the possible spectrum of determinants of binary matrices. For binary $n\times n$-matrices all determinants between zero and the maximal
value can be attained. For $n\ge 7$ gaps occur, see e.g.\ \cite{0735.05017}. The spectrum of the determinants of binary
$7\times 7$-matrices was determined in \cite{0221.05045} to be $\{1,\dots, 18\}\cup\{20,24,32\}$. Using this more  detailed information 
we can conclude that the denominator $q$ of the critical threshold value of a Boolean function with $f(\emptyset)=0$ on $6$ voters is at
most $17$. Thus, we are able to compute a finite superset $\Lambda(n)$ of $Spec_\mathcal{B}(n)$ for each number $n$ of voters. 

Our next aim is to provide an ILP formulation in order to determine the entire spectrum for simple games and complete simple games on $n$
voters. Therefore, we consider  the linear program~(\ref{lp_critical_alpha_value}) for the determination of the critical threshold value.
Dropping the constraint $\alpha\ge 1$ and assuming $w_i\ge 0$ we  abbreviate the emerging { linear program by $\min c^Tx,\,Ax\ge b,x\ge 0$. 
If its optimal value is at least $1$ then it coincides with the critical threshold value. Otherwise the game is weighted. By the
strong duality theorem its dual $\max b^Ty\,A^Ty\le c,y\ge 0$ has the same optimal solution if both are feasible. This is indeed the case
taking the dual solution $y=0$ and primal weights of $1$ with an $\alpha$ of $n$. Thus, we can read of the critical threshold value as $c^Tx$
from each feasible solution of the inequality system $Ax\ge b,\,A^Ty\le c,\, c^Tx=b^Ty,\,x,y\ge 0$.

As done in Section~\ref{sec_ilp_max_alpha} we model the underlying simple game by binary variables $x_S$ for the subsets $S\subseteq N$
and use Big-M constraints:
\begin{eqnarray}
  x_\emptyset &=& 0\label{ie_sg_start}\\
  x_N &=& 1\label{ie_grand_coalition}\\
  x_S-x_{S\backslash\{i\}} &\ge& 0\quad\forall \emptyset\neq S\subseteq N, i\in S\label{ie_monotonicity}\\
  x_S &\in & \{0,1\}\quad\forall S\subseteq N\label{ie_sg_end}\\
  w_i &\le& 1\quad\forall i\in N\label{ie_primal_start}\\
  \sum\limits_{i\in S}w_i &\ge& x_S\quad\forall S\subseteq N\label{ie_implication_3}\\
  \sum\limits_{i\in S} w_i &\le& \alpha +|S|\cdot x_S\quad\forall S\subseteq N\label{ie_implication_4}\\
  w_n &\ge& 0\label{ie_primal_end}
\end{eqnarray}
\begin{eqnarray}    
  \sum\limits_{\{i\}\subseteq S\subseteq N} u_S-\sum\limits_{\{i\}\subseteq S\subseteq N} v_S &\le& 0\quad\forall i\in N\label{ie_dual_start}\\
  \sum\limits_{S\subseteq N} v_S &\le& 1\\
  u_S&\le& x_S\quad\forall S\subseteq N \\
  v_S&\le& 1-x_S \forall S\subseteq N\\
  u_S &\ge & 0\quad\forall S\subseteq N\\
  v_S &\ge & 0\quad\forall S\subseteq N\label{ie_dual_end}\\
  \sum\limits_{S\subseteq N} u_S =\alpha \label{ie_optimality_coupling}
\end{eqnarray}
Inequalities (\ref{ie_sg_start})-(\ref{ie_sg_end}) model the simple games. The primal program to determine the critical
threshold value is given as inequalities (\ref{ie_primal_start})-(\ref{ie_primal_end}). W.l.o.g.\ we can restrict the weights to lie
inside $[0,1]$. Inequality~(\ref{ie_implication_3}) states that the weight of each winning coalition is at least $1$ and that the
weight of each losing coalition is at least zero, which is a valid inequality.  Similarly, Inequality~(\ref{ie_implication_4}) is
fulfilled for $x_S=1$ and translates to $w(S)\le \alpha$ for each losing coalition $S$. The formerly used dual linear program is
stated in inequalities (\ref{ie_dual_start})-(\ref{ie_dual_end}). Finally the coupling of the primal and the dual target value
is enforced in Inequality~(\ref{ie_optimality_coupling}).

We remark that in order to destroy a bit of the inherent symmetry, i.e.\ the group of all permutations on $n$ elements
acts on the set of simple games, we might require $w_1\ge\dots\ge w_n$.

Having this inequality system at hand, one may prescribe each element in $\Lambda(n)$ as a possible value for $\alpha$ and check
whether it is feasible, then $\alpha$ is contained in the spectrum, or not. 

Another possibility to determine the entire spectrum is to solve a sequence of ILPs, where we add the target function $\min\alpha$
and the constraint $\alpha\ge l$. As starting value we choose $l=\min \{v\in\Lambda(n):v>1\}$. If the optimal target value is given
by $\alpha'$, we choose $l=\min \{v\in\Lambda(n):v>\alpha'\}$ until the set is empty. We remark that for larger $n$ the values of
$\Lambda(n)$ might be relatively close to each other so that numerical problems may occur.

Using the latter approach, we have verified the results $Spec_\mathcal{S}(1)=Spec_\mathcal{S}(2)=Spec_\mathcal{S}(3)=Spec_\mathcal{S}(4)=\{1\}$,
$Spec_\mathcal{S}(5)=\left\{1,\frac{6}{5},\frac{7}{6},\frac{8}{7},\frac{9}{8}\right\}$, and $Spec_\mathcal{S}(6)=Spec_\mathcal{S}(5)
\cup\Big\{\frac{3}{2},\frac{4}{3},\frac{5}{4},\frac{9}{7},\frac{10}{9},\frac{11}{9},\frac{11}{10},\frac{12}{11},\frac{13}{10},\frac{13}{11},
\frac{13}{12},\frac{14}{11},\frac{14}{13},$ $\frac{15}{13},\frac{15}{14},\frac{16}{13},\frac{16}{15},\frac{17}{13},\frac{17}{14},\frac{17}{15},
\frac{17}{16}\Big\}$ already given in~\cite{hierarchies}. For $n=7$ we have newly determined the smallest non-trivial critical
threshold value $\min Spec_\mathcal{S}(7)\backslash\{1\} =\frac{40}{39}$. 
\footnote{Since the possible spectrum of determinants is given by $\{0,\dots, 40,42,44,45,48,56\}$, see e.g.\
http://www.indiana.edu/$\sim$maxdet/spectrum.html, only $\frac{45}{44}$ had to be ruled out.} For $n=8$ we conjecture
$\min Spec_\mathcal{S}(8)\backslash\{1\} =\frac{105}{104}$\footnote{Here the possible spectrum of determinants is given by
$\{0,\dots,102,104,105, 108, 110, 112, 116,$ $117, 120, 125, 128,  144\}$ so that only $\frac{117}{116}$ might be possible.}.

By dropping the inequalities (\ref{ie_grand_coalition}), (\ref{ie_monotonicity}) and permitting negative weights, i.e.\ $w_i\in\mathbb{R}$,
we can principally determine the entire spectrum for Boolean functions with $f(\emptyset)=0$. For small $n$, the explicit sets are given by
\begin{eqnarray*}
  Spec_\mathcal{B}(1) &=& \left\{1\right\}\\
  Spec_\mathcal{B}(2) &=& \left\{1,2\right\}\\
  Spec_\mathcal{B}(3) &=& \left\{1,\frac{3}{2},2,3\right\}\\
  Spec_\mathcal{B}(4) &=& \left\{1,\frac{5}{4},\frac{4}{3},\frac{3}{2},\frac{5}{3},2,\frac{5}{2},3,4\right\}\\
  Spec_\mathcal{B}(5) &=& \left\{1,\frac{9}{8},\frac{8}{7},\frac{7}{6},\frac{6}{5},\frac{5}{4},\frac{9}{7},\frac{4}{3},\frac{7}{5},\frac{3}{2},\frac{8}{5},\frac{5}{3},
    \frac{7}{4},\frac{9}{5},2,\frac{9}{4},\frac{7}{3},\frac{5}{2},\frac{8}{3},3,\frac{7}{2},4,5\right\}\\
  Spec_\mathcal{B}(6) &=& \Big\{1,\frac{18}{17},\frac{17}{16},\frac{16}{15},\frac{15}{14},\frac{14}{13},\frac{13}{12},\frac{12}{11},\frac{11}{10},\frac{10}{9},
    \frac{9}{8},\frac{17}{15},\frac{8}{7},\frac{15}{13},\frac{7}{6},\frac{13}{11},\frac{6}{5},\frac{17}{14},\frac{11}{9},\\
    && \frac{16}{13},\frac{5}{4},\frac{14}{11},\frac{9}{7},\frac{13}{10},\frac{17}{13},\frac{4}{3},\frac{15}{11},\frac{11}{8},\frac{18}{13},\frac{7}{5},
    \frac{17}{12},\frac{10}{7},\frac{13}{9},\frac{16}{11},\frac{3}{2},\frac{17}{11},\frac{14}{9},\frac{11}{7},\\
    &&\frac{8}{5},\frac{13}{8},\frac{18}{11},\frac{5}{3},\frac{17}{10},\frac{12}{7},\frac{7}{4},\frac{16}{9},\frac{9}{5},\frac{11}{6},\frac{13}{7},\frac{15}{8},
    \frac{17}{9},2,\frac{17}{8},\frac{15}{7},\frac{13}{6},\frac{11}{5},\frac{9}{4},\\
    &&\frac{16}{7},\frac{7}{3},\frac{12}{5},\frac{5}{2},\frac{13}{5},\frac{8}{3},\frac{11}{4},\frac{14}{5},3,\frac{13}{4},\frac{10}{3},\frac{7}{2},
    \frac{11}{3},4,\frac{9}{2},5,6\Big\}
\end{eqnarray*} 

For complete simple games we simply replace the conditions (\ref{ie_sg_start})-(\ref{ie_sg_end}) by those for complete simples games. As complete
simple games with up to $6$ voters are roughly weighted, we have $Spec_\mathcal{C}(n)=\{1\}$ for $n\le 6$. For $n=7$ we have determined
$\min\Big\{Spec_\mathcal{C}(7)\backslash\{1\}\Big\}=\frac{39}{38}$.

\section{Conclusion}
\label{sec_conclusion}

\noindent
In this paper we have considered the critical threshold values for several subclasses of binary voting structures. For Boolean
games an exact upper bound of $\mu_\mathcal{B}(n)=n$ could be determined. The set of achievable values is strongly related to the spectrum
of determinants of binary matrices, so that Hadamard's bound comes into play.

We have strengthened the lower and upper bound on the maximum critical threshold value of a simple game on $n$ voters to
$\left\lfloor\frac{n^2}{4}\right\rfloor/n\le \s_\mathcal{S}(n)\le\frac{n}{3}$. It remains to prove (or to disprove) the conjecture
that the lower bound is tight. By introducing an integer linear programming approach to determine the maximum critical threshold value
we could algorithmically verify this conjecture for all $n\le 9$. On the one hand, this seems to be a rather small number. On the other hand, regarding
the question of the number of simple games, not much more than a lower bound of $10^{42}$ is known. Since the number of simple games grows doubly
exponential, no huge improvements can be expected from an algorithmic point of view.

For complete simple games the problem to determine $\s_\mathcal{C}(n)$ is considerably harder. The large gap
between the stated upper bound $\frac{c_1n\log\log n}{\log n}$ and lower bound $c_2\sqrt{n}$ deserves to be closed
or at least to be narrowed. In order to facilitate the conjectured asymptotics of $\Theta(\sqrt{n})$ we have provided a class of
examples achieving this bound and have proven the respective upper bounds for several subclasses of complete simple games. 

So far we have no structural insights on those complete simple games which achieve $\s_\mathcal{C}(n)$ as their critical threshold value. The given
integer linear programming formulation for $\s_\mathcal{C}(n)$ made it possible to determine exact values for numbers of voters
where even the number of complete simple games is not known. To be more precise, there are $284\,432\,730\,174$ complete
simple games for nine voters, see e.g.\ \cite{min_sum_rep} or \cite{FrMo10}, while exact numbers are unknown for $n\ge 10$. The fact that the exact numbers for
the critical threshold values $c_\mathcal{C}(n)$ for complete simple games are known up to $n=16$, indicates the great potential of our introduced algorithmic
approach. Similar integer linear programming formulations can possibly be developed for other problems on extremal voting schemes. Applications to related
concepts like, e.g., the nucleolus or the cost of stability seem to be promising.

In this paper we leave the question for the complexity to determine the criticial threshold value within a given class of games open, but expect
it to be in NP in general.

Concerning the discriminability of the hierarchy of $\alpha$-roughly weighted simple games, it would be nice to prove (if true) that there is a complete
simple game $\chi$ with critical threshold value $\mu(\chi)=\frac{p}{q}$ for all integers $p\ge q$. Some first experiments let us conjecture
that there even is a complete simple game with two types of voters and one shift-minimal winning vector.

As usual, the relation to other solution concepts from the game theory literature to the critical threshold value should be studied. We have
started this task by considering the cost of stability. Is turns out that the critical threshold value is upper bounded by the cost of stability. From that,
we  could deduce an upper bound of $\sqrt{n}$ for super-additive games. For Boolean games the asymptotic extremal values coincide, while they can differ
to a large extent for concrete games.

The maximum critical threshold value can discriminate between the classes of simple games, complete simple games, 
and weighted voting games, while the cost of stability can not. The concept of a dimension of a simple game is not directly related to the
critical threshold value. 

The concept of $\alpha$-weightedness seems very interesting. More research should be done in that direction. A quite natural idea is to
transfer the concept to ternary voting games, see e.g.\ \cite{springerlink:10.1007/BF01263275} and \cite{1073.91542}, or graph based games like e.g.\ network
flow games. Also effectivity functions, see e.g.\ \cite{springerlink:10.1007/BF01295853}, might be candidates for a generalization of the
basic concept. Last but not least, there are two additional hierarchies of simple games described in \cite{hierarchies} which deserve to be
analyzed in more detail.

\section*{Acknowledgments}
We would like to thank Tatyana Gvozdeva, Stefan Napel and two anonymous referees for their thoughtful comments that helped very much to improve
the presentation of the paper.

\bibliographystyle{apalike}
\bibliography{alpha_roughly_weighted_games}

\begin{thebibliography}{}

\bibitem[Algaba et~al., 2007]{1110.90067}
Algaba, E., Bilbao, J.~M., and Fern\'andez, J.~R. (2007).
\newblock The distribution of power in the {E}uropean {C}onstitution.
\newblock {\em Eur. J. Oper. Res.}, 176(3):1752--1766.

\bibitem[Bachrach, 2011]{pre05940089}
Bachrach, Y. (2011).
\newblock The least-core of threshold network flow games.
\newblock Murlak, Filip (ed.) et al., Mathematical foundations of computer
  science 2011. 36th international symposium, MFCS 2011, Warsaw, Poland, August
  22--26, 2011. Proceedings. Berlin: Springer. Lecture Notes in Computer
  Science 6907, 36-47 (2011).

\bibitem[Bachrach et~al., 2009]{Bachrach:2009:CSC:1692490.1692502}
Bachrach, Y., Elkind, E., Meir, R., Pasechnik, D., Zuckerman, M., Rothe, J.,
  and Rosenschein, J. (2009).
\newblock The cost of stability in coalitional games.
\newblock In {\em Proceedings of the 2nd International Symposium on Algorithmic
  Game Theory}, SAGT '09, pages 122--134, Berlin, Heidelberg. Springer-Verlag.

\bibitem[Berthold et~al., 2011a]{BertholdGleixnerHeinzVigerske2011TR}
Berthold, T., Gleixner, A.~M., Heinz, S., and Vigerske, S. (2011a).
\newblock On the computational impact of {M}{I}{Q}{C}{P} solver components.
\newblock ZIB-Report 11-01, Zuse Institute Berlin.
\newblock
  \url{http://vs24.kobv.de/opus4-zib/frontdoor/index/index/docId/1199/}.

\bibitem[Berthold et~al., 2011b]{BeHeVi09}
Berthold, T., Heinz, S., and Vigerske, S. (2011b).
\newblock Extending a {CIP} framework to solve {MIQCP}s.
\newblock In Lee, J. and Leyffer, S., editors, {\em Mixed-integer nonlinear
  optimization: Algorithmic advances and applications}, IMA volumes in
  Mathematics and its Applications. Springer.
\newblock to appear.

\bibitem[Brenner and Cummings, 1972]{0249.15003}
Brenner, J. and Cummings, L. (1972).
\newblock The {H}adamard maximum determinant problem.
\newblock {\em Am. Math. Mon.}, 79:626--630.

\bibitem[Carreras and Freixas, 1996]{complete_simple_games}
Carreras, F. and Freixas, J. (1996).
\newblock Complete simple games.
\newblock {\em Math. Soc. Sci.}, 32:139--155.

\bibitem[Carreras and Freixas, 2004]{1087.91005}
Carreras, F. and Freixas, J. (2004).
\newblock A power analysis of linear games with consensus.
\newblock {\em Math. Soc. Sci.}, 48(2):207--221.

\bibitem[Craigen, 1990]{0735.05017}
Craigen, R. (1990).
\newblock The range of the determinant function on the set of $n \times{} n$
  (0,1)- matrices.
\newblock {\em J. Comb. Math. Comb. Comput.}, 8:161--171.

\bibitem[De\u{\i}neko and Woeginger, 2006]{1134.68369}
De\u{\i}neko, V.~G. and Woeginger, G.~J. (2006).
\newblock On the dimension of simple monotonic games.
\newblock {\em Eur. J. Oper. Res.}, 170(1):315--318.

\bibitem[Diakonikolas and Servedio, 2012]{approximation_of_linear_threshold}
Diakonikolas, I. and Servedio, R. (2012).
\newblock Improved approximation of linear threshold functions.
\newblock {\em Computational Complexity}, page 33 p.
\newblock to appear, available at http://arxiv.org/abs/0910.3719.

\bibitem[Felsenthal and Machover, 1997]{springerlink:10.1007/BF01263275}
Felsenthal, D.~S. and Machover, M. (1997).
\newblock Ternary voting games.
\newblock {\em Int. J. Game Theory}, 26:335--351.

\bibitem[Freixas and Molinero, 2010]{FrMo10}
Freixas, J. and Molinero, X. (2010).
\newblock Weighted games without a unique minimal representation in integers.
\newblock {\em Optim. Methods Softw.}, 25:203--215.

\bibitem[Freixas and Puente, 2008]{1151.91021}
Freixas, J. and Puente, M.~A. (2008).
\newblock Dimension of complete simple games with minimum.
\newblock {\em Eur. J. Oper. Res.}, 188(2):555--568.

\bibitem[Freixas and Zwicker, 2003]{1073.91542}
Freixas, J. and Zwicker, W. (2003).
\newblock Weighted voting, abstention, and multiple levels of approval.
\newblock {\em Soc. Choice Welfare}, 21(3):399--431.

\bibitem[Granot and Granot, 1992]{0773.90097}
Granot, D. and Granot, F. (1992).
\newblock On some network flow games.
\newblock {\em Math. Oper. Res.}, 17(4):792--841.

\bibitem[Gvozdeva et~al., 2012]{hierarchies}
Gvozdeva, T., Hemaspaandra, L.~A., and Slinko, A. (2012).
\newblock Three hierarchies of simple games parameterized by ``resource``
  parameters.
\newblock {\em Int. J. Game Theory}, page 17 p.
\newblock to appear, DOI: 10.1007/s00182-011-0308-4.

\bibitem[Gvozdeva and Slinko, 2011]{Gvozdeva201120}
Gvozdeva, T. and Slinko, A. (2011).
\newblock Weighted and roughly weighted simple games.
\newblock {\em Math. Social Sci.}, 61(1):20--30.

\bibitem[Isbell, 1958]{0083.14301}
Isbell, J. (1958).
\newblock A class of simple games.
\newblock {\em Duke Math. J.}, 25:423--439.

\bibitem[Kalai and Zemel, 1982]{0498.90030}
Kalai, E. and Zemel, E. (1982).
\newblock Totally balanced games and games of flow.
\newblock {\em Math. Oper. Res.}, 7:476--478.

\bibitem[Koch, 2004]{Koch2004b}
Koch, T. (2004).
\newblock {\em Rapid Mathematical Programming}.
\newblock PhD thesis, Technische {Universit\"at} Berlin.

\bibitem[Kurz, 2012a]{min_sum_rep}
Kurz, S. (2012a).
\newblock On minimum sum representations for weighted voting games.
\newblock {\em Ann. Oper. Res.}, 196(1):361--369.

\bibitem[Kurz, 2012b]{inverse_power_index_problem}
Kurz, S. (2012b).
\newblock On the inverse power index problem.
\newblock {\em Optimization}, 16(8):989--1011.

\bibitem[Letchford and Galli, 2011]{reformulation}
Letchford, A.~N. and Galli, L. (2011).
\newblock Reformulating mixed-integer quadratically constrained quadratic
  programs.
\newblock {\em SIAM J. Opt.}
\newblock 23 pages, submitted, available at
  http://www.optimization-online.org/DB\_HTML/2011/02/2919.html.

\bibitem[Metropolis, 1971]{0221.05045}
Metropolis, N. (1971).
\newblock Spectra of determinant values in (0,1) matrices.
\newblock Computers in Number Theory, Proc. Atlas Sympos. No.2, Oxford 1969,
  271--276.

\bibitem[Peleg, 1992]{0856.90029}
Peleg, B. (1992).
\newblock Voting by count and account.
\newblock Selten, Reinhard (ed.), Rational interaction. Essays in honor of John
  C. Harsanyi. Berlin: Springer-Verlag. 45--51.

\bibitem[Resnick et~al., 2009]{pre05617303}
Resnick, E., Bachrach, Y.and~Meir, R., and Rosenschein, J. (2009).
\newblock The cost of stability in network flow games.
\newblock Kr\'alovi\v c, Rastislav (ed.) et al., Mathematical foundations of
  computer science 2009. 34th international symposium, MFCS 2009, Novy
  Smokovec, High Tatras, Slovakia, August 24--28, 2009. Proceedings. Berlin:
  Springer. Lecture Notes in Computer Science 5734, 636-650 (2009).

\bibitem[Storcken, 1997]{springerlink:10.1007/BF01295853}
Storcken, T. (1997).
\newblock Effectivity functions and simple games.
\newblock {\em Int. J. Game Theory}, 26:235--248.

\bibitem[Taylor and Zwicker, 1993]{0765.90030}
Taylor, A. and Zwicker, W. (1993).
\newblock Weighted voting, multicameral representation, and power.
\newblock {\em Games Econ. Behav.}, 5(1):170--181.

\bibitem[Taylor and Zwicker, 1999]{0943.91005}
Taylor, A.~D. and Zwicker, W.~S. (1999).
\newblock {\em Simple games. Desirability relations, trading,
  pseudoweightings}.
\newblock Princeton, NJ: Princeton University Press. 246 p.

\bibitem[Tijs, 2011]{1229.91004}
Tijs, S. (2011).
\newblock {\em Introduction to game theory}.
\newblock Texts and Readings in Mathematics 23. New Dehli: Hindustan Book
  Agency. viii, 176~p.

\bibitem[Vanderbei, 2008]{vanderbei}
Vanderbei, R. (2008).
\newblock {\em Linear programming. Foundations and extensions. 3rd ed.}
\newblock International Series in Operations Research \&amp; Management Science
  114. New York, NY: Springer. xix, 464~p.

\bibitem[von Neumann and Morgenstern, 2007]{1112.91002}
von Neumann, J. and Morgenstern, O. (2007).
\newblock {\em Theory of games and economic behavior. With an introduction by
  Harold Kuhn and an afterword by Ariel Rubinstein. 4th print of the 2004
  sixtieth-anniversary ed.}
\newblock Princeton, N J: Princeton University Press. xxxii, 739~p.

\end{thebibliography}

\appendix
\section{Further side results}

\noindent
In this appendix we mention some additional results, which are obtained with the techniques described in the paper, 
but are a bit to specific to be included in the main part.

\subsection{Strong or proper simple games}

\noindent
In Section~\ref{sec_ilp_max_alpha} we have mentioned that one can easily model restrictions within the class of simple games, 
e.g.\ consider proper or strong simple games. So, for each voting class $\mathcal{X}\in\{\mathcal{B},\mathcal{S},\mathcal{C}\}$
let $\s_{\mathcal{X}}^s(n)$ denote the maximum critical threshold value of a game consisting of $n$ voters in $\mathcal{X}$,
which is strong. Similarly, we define $\s_{\mathcal{X}}^p(n)$ for games which are proper and $\s_{\mathcal{X}}^{ps}(n)$ for
games which are proper and strong. Numerical results for small numbers of voters are stated in Table~\ref{table_ar_proper_strong_csg}.  

\begin{table}[htp]
\begin{center}
  \begin{tabular}{ccccc}
  \hline
  $\mathbf{n}$ & $\mathbf{\s_{\mathcal{C}}(n)}$ & $\mathbf{\s_{\mathcal{C}}^p(n)}$ &
  $\mathbf{\s_{\mathcal{C}}^s(n)}$ & $\mathbf{\s_{\mathcal{C}}^{ps}(n)}$ \\[1mm]
   7 & $\frac{8}{7}\approx 1.142857$ & $\frac{14}{13}\approx 1.076923$ & $\frac{10}{9}=1.\overline{1}$ & 1 \\[1mm]
   8 & $\frac{26}{21}\approx 1.238095$ & $\frac{38}{33}=1.\overline{15}$ & $\frac{26}{21}\approx 1.238095$ & $1$ \\[1mm]
   9 & $\frac{4}{3}=1.\overline{3}$ & $\frac{6}{5}=1.2$ & $\frac{4}{3}=1.\overline{3}$ & $\frac{13}{12}=1.08\overline{3}$ \\[1mm]
  10 & $\frac{38}{27}=1.\overline{407}$ & $\frac{66}{53}\approx 1.245283$ & $\frac{38}{27}=1.\overline{407}$ & $\frac{23}{20}=1.15$ \\[1mm]
  11 & $\frac{22}{15}=1.4\overline{6}$ & $1.290735$ & $\frac{22}{15}=1.4\overline{6}$ & $\frac{43}{36}=1.19\overline{4}$ \\[1mm]
  12 & $\frac{14}{9}=1.\overline{5}$ & $\frac{4}{3}=1.\overline{3}$ & $1.553571$ & $\frac{59}{48}=1.2291\overline{6}$\\[1mm]
  13 & $\frac{33}{20}=1.65$ & $\in[1.3620,1.4211]$ & $\frac{33}{20}=1.65$ & $\approx1.258772$\\[1mm]
  14 & $\frac{111}{64}=1.734375$ & & $\frac{111}{64}=1.734375$ & $\approx 1.298361$ \\[1mm]
  \hline
  \end{tabular}
  \caption{The maximum critical threshold value for complete simple games restricted to strong or proper games.}
  \label{table_ar_proper_strong_csg}
\end{center}
\end{table}

Obviously we have the inequalities $\s_{\mathcal{C}}^{ps}(n)\le \s_{\mathcal{C}}^{p}(n)\le \s_{\mathcal{C}}(n)$
and $\s_{\mathcal{C}}^{ps}(n)\le \s_{\mathcal{C}}^{s}(n)\le \s_{\mathcal{C}}(n)$. Since adding an additional player
to an arbitrary complete simple game, which is winning on its own, yields a strong complete simple game with equal
critical threshold value, we also have $\s_{\mathcal{C}}^{s}(n)\ge \s_{\mathcal{C}}(n-1)$, i.e.\ 
Conjecture~\ref{conjecture_spectrum_c_largest_complete_simple_game} would imply 
$\s_\mathcal{C}^{s}(n)\in\Theta\!\left(\sqrt{n}\right)$. Looking at the numerical values of Table~\ref{table_ar_proper_strong_csg}
one might conjecture $\s_{\mathcal{C}}^p(n)\le \s_{\mathcal{C}}^s(n)$ for all $n$. It would be very nice to have a good lower
bound construction for $\s_{\mathcal{C}}^{ps}(n)$, which then would imply lower bounds for $\s_{\mathcal{C}}^{p}(n)$, 
$\s_{\mathcal{S}}^{ps}(n)$, and $\s_{\mathcal{S}}^{p}(n)$.  

\begin{table}[htp]
\begin{center}
  \begin{tabular}{ccccc}
  \hline
  $\mathbf{n}$ & $\mathbf{\s_{\mathcal{S}}(n)}$ & $\mathbf{\s_{\mathcal{S}}^p(n)}$ &
  $\mathbf{\s_{\mathcal{S}}^s(n)}$ & $\mathbf{\s_{\mathcal{S}}^{ps}(n)}$ \\[1mm]
   5 & $\frac{6}{5}=1.2$ & $1$ & $1$ & $1$ \\[1mm]
   6 & $\frac{3}{2}=1.5$ & $\frac{4}{3}=1.\overline{3}$ & $\frac{3}{2}=1.5$ & $1$ \\[1mm]
   7 & $\frac{12}{7}\approx 1.714286$ & $\frac{7}{5}=1.4$ & $\frac{5}{3}=1.\overline{6}$ & $\frac{4}{3}=1.\overline{3}$ \\[1mm]
   8 & $2$ & $\frac{3}{2}=1.5$ & $2$ & $\frac{7}{5}=1.4$ \\[1mm]
   9 & $\frac{20}{9}=2.\overline{2}$ & $\frac{5}{3}=1.\overline{6}$ & $\frac{11}{5}=2.2$ & $\frac{3}{2}=1.5$ \\[1mm]
  \hline
  \end{tabular}
  \caption{The maximum critical threshold value for simple games restricted to strong or proper games.}
  \label{table_ar_proper_strong_sg}
\end{center}
\end{table}

\begin{lemma}
\label{car_simple_strong_even}
For all $k\ge 2$ we have $\s_{\mathcal{S}}^s(2k)\ge\frac{k}{2}$.
\end{lemma}
\begin{proof}
  Consider the $k$ coalitions $S_i:=\{2i-1,2i\}$ for $1\le i\le k$ and the $2^k$ coalitions $\{a_1,\dots,a_k\}$ with
  $a_i\in S_i$. Let us denote the latter set of coalitions by $\mathcal{A}$. We can easily check that those coalitions
  form an antichain so that we can arbitrarily prescribe for   each coalition whether it is winning or losing and there
  exist at least one simple game $\chi$ meeting those conditions.  Here we require that the coalitions $S_i$ are winning
  the coalitions $\{a_1,\dots,a_k\}$ are winning if and only if $a_1=1$, $a_2=3$ or $a_1=2$, $a_2=4$. Since the coalitions
  $S_i$ are winning we have $$k\le\sum_{i=1}^k w(S_i)=\sum_{i=1}^n w_i.$$ Since the coalitions in $\mathcal{A}\cap L$, where $L$
  denotes the set of losing coalitions, contain each voter with equal frequency, we have
  $$
    2^{k-1}\alpha\sum_{A\in\mathcal{A}\cap L} w(A) =\left(\sum_{i=1}^n\right)\cdot\frac{2^k}{2}\cdot \frac{k}{2k}.
  $$   
  Combining both inequalities gives $\alpha\ge\frac{k}{2}$.
\end{proof}

We remark that we have $\s_{\mathcal{S}}^s(n)\le \s_{\mathcal{S}}(n)$ so that the bound from Lemma~\ref{car_simple_strong_even}
is tight if Conjecture~\ref{conjecture_spectrum_c_largest} is true.

\begin{lemma}
\label{car_simple_strong_odd}
For all $k\ge 2$ we have $\s_{\mathcal{S}}^s(2k+5)\ge 1+\frac{k(k+1)}{2k+1}$.
\end{lemma}
\begin{proof}
  We will construct a class of examples by prescribing for some coalitions whether they are winning or losing. For
  $1\le i\le 2k$ we require that the coalitions $\{i,i+1\}$ are winning. Let $B:=\{2i-1\mid 1\le i\le k+1\}$ and 
  $R:=\{2i\mid 1\le i\le k\}$. Next we require
  \begin{eqnarray*}
    \{2k+2,2k+4\}\cup B \in L && \{2k+3,2k+5\}\cup R\in W,\\ 
    \{2k+3,2k+5\}\cup B \in L && \{2k+2,2k+4\}\cup R\in W,\\
    \{2k+2,2k+5\}\cup B \in W && \{2k+3,2k+4\}\cup R\in L,\\
    \{2k+3,2k+4\}\cup B \in W && \{2k+2,2k+5\}\cup R\in L,
  \end{eqnarray*} 
  where $L$ denotes the set of losing coalitions and $W$ denotes the set of losing coalitions. The linear program
  for the computation of the critical threshold value restricted on the mentioned coalitions has an optimal solution
  of $1+\frac{k(k+1)}{2k+1}$.
\end{proof}

We remark that the lower bound from Lemma~\ref{car_simple_strong_odd} misses the value from
Conjecture~\ref{conjecture_spectrum_c_largest} only by $\frac{1}{n(n-4)}$. Since the computed exact values for 
$\s_{\mathcal{S}}(n)$ from Table~\ref{table_ar_proper_strong_sg} coincide with the lower bounds from
Lemma~\ref{car_simple_strong_even} and Lemma~\ref{car_simple_strong_odd}, we conjecture that they are tight. 

Unfortunately we can not use duality to obtain upper bounds for proper simple games from those for strong simple games.
To this end let us consider the class of examples from the proof of Lemma~\ref{car_simple_strong_even}. We observe that
all coalitions of cardinality at least $k+1$ are winning so that each winning coalition of the dual game, which is strong,
has a cardinality of at least $k$. Thus we may choose weights $w_i=\frac{1}{k}$ for all voters so that the weight of each
losing coalition is at most $2$ while the original game has a critical threshold value of $\max(1,\frac{k}{2})$.

\begin{table}[htp]
\begin{center}
  \begin{tabular}{ccccc}
  \hline
  $\mathbf{n}$ & $\mathbf{\s_{\mathcal{B}}(n)}$ & $\mathbf{\s_{\mathcal{B}}^p(n)}$ &
  $\mathbf{\s_{\mathcal{B}}^s(n)}$ & $\mathbf{\s_{\mathcal{B}}^{ps}(n)}$ \\[1mm]
   1 & $1$ & $1$ & $1$ & $1$ \\[1mm]
   2 & $2$ & $1$ & $1$ & $1$ \\[1mm]
   3 & $3$ & $3$ & $2$ & $2$ \\[1mm]
   4 & $4$ & $4$ & $3$ & $3$ \\[1mm]
   5 & $5$ & $5$ & $4$ & $4$ \\[1mm]
   6 & $6$ & $6$ & $5$ & $5$ \\[1mm]
   7 & $7$ & $7$ & $6$ & $6$ \\[1mm]
   8 & $8$ & $8$ & $7$ & $7$ \\[1mm]
   9 & $9$ & $9$ & $8$ & $8$ \\[1mm]
  \hline
  \end{tabular}
  \caption{The maximum critical threshold value for Boolean games restricted to strong or proper games.}
  \label{table_ar_proper_strong_bg}
\end{center}
\end{table}

\begin{lemma}
  For $n\ge 3$ we have $\s_{\mathcal{B}}^p(n)=n$. 
\end{lemma}
\begin{proof}
  Of course we have $\s_{\mathcal{B}}^p(n)\le \s_{\mathcal{B}}(n)=n$. A proper example achieving this bound
  is given by the Boolean game whose winning coalitions coincide with the coalitions of size one.
\end{proof}

\begin{lemma}
We have $\s_{\mathcal{B}}^s(n)=\max(1,n-1)$ for all $n\in\mathbb{N}$. 
\end{lemma}
\begin{proof}
Since the empty set is a losing coalition, its complement, the grand coalition, has to be winning. Thus every losing coalition
consists of at most $n-1$ members. Choosing weights $w_i=1$ for all voters gives a feasible weighting with $\alpha\le n-1$. For
the other direction consider the strong \textit{game} in $\mathcal{B}_n$ with $n\ge 3$, whose losing coalitions are the empty set
and the coalitions of size $n-1$. Since all coalitions of size $1$ are winning, the weights of the players have to be at least
one so that the losing coalitions of cardinality $n-1$ have a weight of at least $n-1$. 
\end{proof}

\begin{lemma}
We have $\s_{\mathcal{B}}^{ps}(n)=\max(1,n-1)$ for all $n\in\mathbb{N}$. 
\end{lemma}
\begin{proof}
  Since $\s_{\mathcal{B}}^{ps}(n)\le\s_{\mathcal{B}}^{s}(n)=\max(1,n-1)$ it suffice to construct an example whose
  critical threshold value reaches the upper bound. To this end we define the strong and proper Boolean game $\chi$
  for $n\ge 3$ as follows: The empty coalition is loosing, the grand coalition is winning, coalitions with sizes between one
  and $\frac{n-1}{2}$, coalitions with sizes between $\frac{n+1}{2}$ and $n-1$ are winning, and coalitions of cardinality 
  $\frac{n}{2}$ are winning if and only if they contain voter~$1$.   
\end{proof}

\subsection{Restrictions on the number of shift-minimal winning vectors}

\noindent
Using the described ILP approach we may also exactly determine the maximal alpha-values $\s_{\mathcal{C}}(n,1)$ of complete
simple games with $n$ players and a single shift-minimal winning coalition. As all complete simple games with at most
six voters are roughly weighted we have $\tilde{s}(n,1)=1$ for $n=6$. The next exact values are given by
\begin{itemize}
 \item $\s_{\mathcal{C}}(7,1)=\frac{10}{9}\approx 1.111111$: (2, 5); (1, 2); (2, 0), (0, 5)
 \item $\s_{\mathcal{C}}(8,1)=\frac{6}{5}=1.2$: (2, 6); (1, 2); (2, 0), (0, 6)
 \item $\s_{\mathcal{C}}(9,1)=\frac{15}{11}\approx 1.272727$: (2, 7); (1, 2); (2, 0), (0, 7)
 \item $\s_{\mathcal{C}}(10,1)=\frac{4}{3}\approx 1.333333$: (2, 8); (1, 2); (2, 0), (0, 8)
 \item $\s_{\mathcal{C}}(11,1)\approx 1.41176470588$: (3, 8); (1, 3); (3, 0), (0, 8)
 \item $\s_{\mathcal{C}}(12,1)=\frac{3}{2}=1.5$: (3, 9); (1, 3); (3, 0), (0, 9)
 \item $\s_{\mathcal{C}}(13,1)\approx 1.57894736842$: (3, 10); (1, 3); (3, 0), (0, 10) 
 \item $\s_{\mathcal{C}}(14,1)=\frac{33}{20}=1.65$: (3, 11); (1, 3); (3, 0), (0, 11) 
\end{itemize}
Here we also state the cardinality vector, the list of shift-minimal winning vectors, and the list of shift-maximal
losing vectors of an example reaching the upper bound $\s_{\mathcal{C}}(n,1)$, respectively. 

\bigskip

We can enhance our ILP formulations to additionally treat conditions on the shift-minimal winning coalitions easily.
For $S\subseteq N$ we introduce a binary variable $s_S$ with the meaning that $s_S=1$ iff coalition $S$ is a shift
minimal winning coalition. As conditions we have
\begin{eqnarray*}
  s_S\le x_S\\
  s_S\le 1-x_{S'}\quad \forall S'\prec S:\nexists S'':S'\prec S''\prec S\\
  -x_S+\sum\limits_{S'\prec S:\nexists S'':S'\prec S''\prec S} x_{S'}\,+\, s_S \ge 0.
\end{eqnarray*}
By setting
$$
  \sum\limits_{S\subseteq N} s_S=r
$$
we can easily formulate exact values, lower or upper bounds for the number $r$ of shift-minimal winning coalitions.
To be able to express the number $t$ of equivalence classes of voters we introduce the functions
$\varphi_i:2^N\rightarrow\{0,1\}$ for all $1\le i\le n-1$ where $\varphi_i(S)=1$ iff a shift-minimal winning
coalition $S$ implies that voter $i$ and voter $i+1$ have to be in different equivalence classes. We use binary
variables $p_i$ for $1\le i\le n-1$ and the constraints
\begin{eqnarray*}
  p_i\ge s_S \cdot\varphi_i(S)\quad\forall S\subseteq N,1\le i\le n-1\\
  p_i\le \sum\limits_{S\subseteq N} s_S\cdot\varphi_i(S)\quad\forall 1\le i\le n-1\\
  \sum\limits_{i=1}^{n-1} p_i=t.
\end{eqnarray*}
Let us denote by $\s_{\mathcal{C}}(n,r,t)$ the maximum critical $\alpha$-value of a complete simple game with $r$
shift-minimal winning coalitions consisting of $n$ voters being partitioned into $t$ equivalence classes. We have
$\s_{\mathcal{C}}(7,1,2)=\frac{10}{9}$, $\s_{\mathcal{C}}(7,2,2)=\frac{17}{15}$,
$\s_{\mathcal{C}}(7,3,2)=\frac{8}{7}$, $\s_{\mathcal{C}}(7,4,2)=\frac{15}{15}$, $\s_{\mathcal{C}}(7,5,2)=\frac{19}{17}$,
and there are no such games for $r\ge 6$. Examples of the corresponding sets of the shift-minimal winning coalitions are given
by $\{35\}$, $\{41,70\}$, $\{44,49,67\}$, $\{43,44,49,67\}$, and $\{31,60,86,88,96\}$, respectively.


\section{Comparision of different ILP solvers}
We give some running time information for different ILP solvers in Table~\ref{table_compare_cplex_vs_gurobi}.

\begin{table}[htp]
  \begin{center}
    \begin{tabular}{rrrrrrrrr}
    \hline\hline
        & \multicolumn{2}{c}{CPLEX} & \multicolumn{2}{c}{CPLEX$^\star$}& \multicolumn{2}{c}{Gurobi~4.0.0} & \multicolumn{2}{c}{Gurobi~4.5.0}\\
    $n$ & nodes & seconds & nodes & seconds & nodes & seconds & nodes & seconds\\  
    \hline
     7 &   459 &    0.4 &  1113 &  0.7 &    975 &    0.5 &    582 &   0.4 \\
     8 &  3721 &     10 &  2271 &  3.7 &   1900 &    1.7 &   1715 &   1.8 \\
     9 &  3594 &     25 &  3297 &   14 &   3153 &     15 &   3724 &    12 \\
    10 & 11799 &    154 &  8974 &   94 &  12008 &    123 &  20988 &    83 \\
    11 & 33312 &   2052 & 42340 & 2131 &  29049 &    349 & 102306 &   979 \\
    12 & 55180 &\it{32379}&     &      &  45752 &   1301 & 215336 &  5403 \\
    13 & 94982 &\it{304255}&    &      &  64962 &   4318 &  83393 & 20408 \\
    14 &       &        &       &      &  97532 &  22230 &        &       \\
    15 &       &        &       &      & 152047 & 134118 &        &       \\
    16 &       &        &       &      & 308240 & 230964 &        &       \\
    \hline\hline
    \end{tabular}
    \caption{Comparing different ILP solvers (using 4~available kernels).}
    \label{table_compare_cplex_vs_gurobi}
  \end{center}
\end{table}
 
 The solvers CPLEX~12.1.0 and Gurobi~4.0.0 are used with the standard parameter settings. Using the tuning option of CPLEX we find
 out that the parameter settings \texttt{mip strategy heuristicfeq -1}, \texttt{mip strategy probe -1}, and \texttt{mip strategy
 variableselect 4} might be better suited. The results are summarized under column CPLEX$^\star$ of Table~\ref{table_compare_cplex_vs_gurobi}.
 We may say that these parameter settings might be good for small instances but can not be generalized to
 larger instances easily.

\end{document}